\documentclass[smallcondensed]{svjour3}     

\smartqed  
\usepackage{graphicx}
\usepackage{url}
\usepackage{amsmath}
\usepackage{amsfonts}
\usepackage{amssymb}
\usepackage{amsopn}
\usepackage{algorithm}
\usepackage{color}
\usepackage{algorithmic}
\usepackage{cite}
\usepackage{authblk}
\usepackage{hyperref}
\hypersetup{
}
\usepackage{mydefs}
\begin{document}
\raggedbottom

\title{Persistent and Zigzag Homology
\thanks{GC and BN were supported by Altor Equity Partners AB through Unbox AI (\url{unboxai.org}).  BN was also supported by the US Department of Energy, Contract DE-AC02-76SF00515, and the Defense Advanced Research Projects Agency (DARPA) under Agreement No. HR00112190040. AD was supported by Burt and Deedee McMurtry Fellowship.  DoD distribution statement: Approved for public release; distribution is unlimited.}
}
\subtitle{A Matrix Factorization Viewpoint}


\author{Gunnar Carlsson        \and
        Anjan Dwaraknath       \and
        Bradley J.~Nelson
}


\institute{G. Carlsson \at
              Department of Mathematics\\
              Stanford University          
           \and
           A. Dwaraknath \at
              Institute for Computational and Mathematical Engineering\\
              Stanford University\\
              \emph{Present address: Cerebras Systems Inc.}
            \and
            B. Nelson \at
                Institute for Computational and Mathematical Engineering\\ 
                Stanford University\\
                \email{bradnelson@uchicago.edu}\\
                \emph{Present address: Department of Statistics, University of Chicago}
}

\date{\today}

\maketitle

\begin{abstract}
Over the past two decades, topological data analysis has emerged as a field of applied mathematics with new applications and algorithmic developments appearing rapidly.  Two fundamental computations in this field are persistent homology and zigzag homology.  In this paper, we show how these computations in the most general case reduce to finding a canonical form of a matrix associated with a type A quiver representation, which in turn can be computed using factorizations of associated matrices.  We show how to use arbitrary induced maps on homology for computation, providing a framework that goes beyond the capabilities of existing software for topological data analysis.  Furthermore, this framework offers multiple opportunities for parallelization which have not been previously exploited.  We provide several examples of the utility of this framework, demonstrate parallel speedups, and report on significant improvements in comparison to existing software.
\end{abstract}

\section{Introduction}\label{sec:intro}

For most of its existence, algebraic topology has been a subject in which computations have been carried out entirely ``by hand".  The fundamental reason for this fact is that singular homology, the structure for computing algebraic invariants  directly on a space without any chosen triangulation, performs linear algebra on vector spaces of uncountably infinite dimension, and is therefore not suitable for any  direct computational method.  As a consequence, numerous methods for computing homology indirectly, such as the Mayer-Vietoris sequence, the long exact sequence of a pair, the nerve lemma,  the K\"{u}nneth formula, and the Serre spectral sequence have been developed.  The methods are so powerful that they can begin with the computation of the homology of a point as input (which is trivial) and arrive at the homology of any space of the homotopy type of a simplicial complex.  The key tool is the notion of {\em functoriality}, first identified by Emmy Noether (see \cite{Hilton88}), which consists of the idea that topological invariants should not be treated as numbers but as algebraic objects such as groups and vector spaces, and crucially that maps between spaces induce homomorphisms of these objects.    Computational topology, on the other hand, has operated mostly by direct computation on the chain level.  These computations are often very expensive, both in time and space.  The indirect methods that are required even to make computations possible in traditional topology can be used in  applied and computational topology  to make computations much smaller and faster, often by making parallelization possible.  In addition, though, functoriality is used in the creation of invariants of that in an appropriate sense measure the shape of the point cloud.  Persistent homology \cite{Robins99} is the first example of this idea, but much work has been done in extending it, notably to {\em zigzag persistence} \cite{ZZtheory2010}.  In this paper, we adapt and implement methods built on functoriality and evaluate the improvements they enable in some key situations.

\subsection{Contributions and Related Work}

Our main contributions are:
\begin{enumerate}
\item{Systematic methods for computing induced homomorphisms on homology with field coefficients. These methods are implemented in software, and interpreted using matrix factorization methods which are described in detail.} 
\item{Sequential and parallel algorithms to extract interval indecomposables (barcodes) from any finite type A quiver representation over any field.  These are also implemented in software using matrix factorization methods and described in detail.}
\item{Application of these methods to the computation of persistent and zigzag homology.}
\item{Evaluation of improvements in speed made possible by our tools.}
\end{enumerate} 

While we provide concrete algorithms for all steps, we also operate at a high level of abstraction which will enable optimizations to be applied in future work. 

The study of persistent and zigzag homology as examples of type A quiver representations was begun in \cite{ZZtheory2010}, which is the theoretical basis for our algorithm.  This has led a variety of interesting theoretical and applied work as surveyed in \cite{Oudot}, but algorithmic implementations so far have not significantly leveraged this connection.  

An algorithm for computing persistent homology was first described in \cite{edelsbrunner2000topological} (for $\FF_2$ coefficients) and was extended to general fields in \cite{ZCComputingPH2005}, and an algorithm for computing zigzag homology was first described in \cite{ZZalg2009}.  These approaches operate on inclusion maps between spaces, and computations work directly on chain complex boundary matrices.  While both persistent and zigzag homology are known to be computable in matrix-multiplication time in the number of cells in a filtration \cite{ZZmatmultime2011}, sparsity considerations are typically much more important in practice, and in order to compute on large data sets, several approaches have been pursued.  First, there have been efforts to speed up computation of persistence through various optimizations \cite{desilvaDualitiesPersistentCo2011, chenPersistentHomologyComputation2011, bauerClearCompressComputing2014} and high performance implementations \cite{GUDHI15, Ripser19, Dionysus2, Eirene16, Phat2017, HYPHA19}.  Second, there have been efforts to reduce the inherent size of computations using methods that preserve the homotopy type of a space while reducing the size of its combinatorial representation \cite{mischaikowMorseTheoryFiltrations2013, dlotkoSimplificationComplexesPersistent2014, wilkersonComputingPersistentFeatures2014,  boissonnatStrongCollapsePersistence2018}.  Zigzag homology has received less attention than persistence, but similar efforts can be found in \cite{mariaZigzagPersistenceReflections2014, mariaDiscreteMorseTheory2019}.  The use of non-inclusion maps in persistent and zigzag homology has been somewhat limited in topological data analysis, although the case of simplicial maps has been investigated in \cite{DeySimplicalMap, kerberBarcodesTowersStreaming2019}, based on a strategy that uses zigzag homology to compute a persistence barcode, and the implementation of zigzag homology in Javaplex \cite{Javaplex} contains tools to compute induced maps for the bivariate witness construction \cite{ZZtheory2010, tausz2012}.

Our approach has several notable differences compared to existing computational approaches for persistent and zigzag homology.  First, we consider a two step approach, where first induced maps on homology are computed to form a quiver representation, and then we compute the barcode.  In contrast, existing approaches work almost exclusively on the level of chain complexes, missing out on the abstraction and compression that induced maps on homology afford.  Second, our approach works for general cell complexes, and general cell maps, and arbitrary fields, whereas some existing approaches are focused on the simplicial (or cubical) complexes and simplicial maps or are limited to $\FF_2$ coefficients \cite{DeySimplicalMap}.  Third, our approach offers multiple opportunities for parallelization, whereas existing implementations to zigzag homology are sequential in nature.  Computing induced maps on homology is trivially parallelizable, and our quiver algorithm also admits a divide and conquer parallelization scheme.  The first divide and conquer approach for zigzag homology was described at a high level in \cite{ZZalg2009}, which operates on chain complexes instead of induced maps on homology, and has not been implemented to the best of our knowledge.  Another approach observed the embarassingly parallel computation of induced maps, and provided a divide and conquer scheme for induced maps based on pullbacks \cite{skrabaParallelScalableZigzag}, but again these observations were not implemented in any parallel framework.  Greg Henselman has also had thoughts on a divide and conquer approach for zigzag homology \cite{henselman2017talk}. A scheme to simplify complexes that is trivially parallelizable for reasons similar to the trivial parallelization of induced maps is found in \cite{boissonnatStrongCollapsePersistence2018}.   Our parallelization scheme is different and complementary to existing efforts to use spectral sequences to parallelize homology calculations \cite{lewisParallelComputationPersistent2015, yoon2018}, as well as parallelization used in clearing and compression optimizations \cite{bauerClearCompressComputing2014}.  Finally, we unify computations for persistent and zigzag homology to a degree that is not seen in existing algorithms, in the sense that modifications to handle arrows of different directions are trivial when stated in terms of matrix factorizations.

Our approach is closest in spirit to the original paper on zigzag homology \cite{ZZtheory2010}, which started with induced maps on homology, and gave a constructive algebraic algorithm for the interval indecomposables for type A quiver representations.  In this paper, we address the computational questions that are necessary for computing arbitrary induced maps on homology, and an explicit algorithm for computing interval indecomposables.  In contrast to \cite{ZZtheory2010} and existing algorithms for zigzag homology \cite{Javaplex, Dionysus2, mariaZigzagPersistenceReflections2014, mariaDiscreteMorseTheory2019}, our algorithm does not explicitly use right filtrations, and instead uses an approach that involves a matrix factorization.  Available zigzag homology implementations can be found in Javaplex \cite{Javaplex}, which is based on the algebraic algorithm in \cite{ZZtheory2010}, and Dionysus \cite{Dionysus2}, which is based on the algorithm in \cite{ZZalg2009}, but neither of these implementations employs parallelism.

Portions of this work can also be found in the Ph.D. dissertations of two co-authors \cite{nelsonThesis2020, dwaraknathThesis2020}.  In this paper, we provide further complexity bounds and computational experiments, particularly for our divide-and-conquer algorithm.

As a companion to this work, we have released a new open-source package for the algorithms we describe.  Inspired by the basic linear algebra subprograms (BLAS) \cite{BLAS1979}, we call it the basic applied topology subprograms (BATS).  The library includes high-level C++ templates for the algorithms described, as well as compatible matrix implementations which are templated over a choice of field.  The code is freely available at \url{https://github.com/bnels/BATS}.

\subsection{Outline}

The rest of the paper is organized as follows:
We will review background necessary for computing homology and induced maps in \cref{sec:prelim}, along with relevant background on computing persistent and zigzag homology and a more formal treatment of classification of quiver representations and its connection to topological data analysis.
In \cref{sec:quiveralg} we will describe our sequential and parallel algorithms for computing indecomposables of A-type quiver representations, and give bounds on their complexity.  
In \cref{sec:applications} we provide several examples of how our methodology and software may be used, and demonstrate parallel speedups in our implementation.  We also compare the performance of our implementation to the zigzag homology implementation in Dionysus \cite{Dionysus2}, demonstrating significant performance improvements.

\section{Preliminaries}\label{sec:prelim}
We will now introduce the algebraic and topological tools necessary for our algorithm.  Specifically, we will need to compute homology of cell complexes, and compute induced maps on homology, and obtain quiver representations from diagrams of spaces.

 For a general reference for matrix computations, we defer to Golub and Van Loan's text \cite{GVL}.  For concepts in algebraic topology, we refer to Hatcher \cite{HatcherAT}.  Early sources for computing persistent homology and zigzag homology are \cite{edelsbrunner2000topological, ZCComputingPH2005} and \cite{ZZtheory2010, ZZalg2009} respectively.  Further background on computational topology can be found in \cite{edelsbrunnerHarerBook2010} and for an overview of modern TDA software we recommend \cite{otterRoadmapComputationPersistent2017}.

\subsection{Algebra and Notation}

{\bf Notation:} Our algorithm uses matrix factorizations to organize computation, and so we shall use Householder notation for linear algebra \cite{GVL, Householder}.  Upper case Greek or Roman letters such as $A$ or $\Lambda$ will refer to matrices, lower case Roman letters such as $a$ will refer to vectors, and lower case Greek letters such as $\lambda$ will refer to scalars.  This is not always consistent with notation found in algebraic topology or the TDA literature, which does not have conventions as strong as Householder notation, but we will attempt to use notation that is close to modern use.  One special symbol, $\partial$, will always refer to a boundary matrix.  Topological spaces will be denoted with calligraphic font as in  $\calX$.

We will prefer to work with matrix factorizations when possible.  Due to the abundance of subscripts in other contexts, we will use square brackets instead of subscripts for indexing.  Vectors will be assumed to be column vectors, and $v^T$ will denote the corresponding row vector.  $e_i$ will refer to the vector with $1$ in the $i$-th entry and $0$ elsewhere, with dimension determined by context.  When we need to access elements in vectors, we will use the notation $x[i]$ to denote the scalar value $e_i^T x$, or the $i$-th entry in $x$.  When we need to access elements of matrices, we will use the notation $A[i,j]$ to denote the scalar $e_i^T A e_j$, or the entry in the $i$-th row and $j$-th column of $A$.  When we want to indicate columns of matrices, we will use the notation $A[j]$ to denote the vector $A e_j$, or the $j$-th column of $A$.  As is standard in numerical linear algebra, indexes will begin with $1$, meaning the valid range of indexes for a vector $x\in \FF^n$ is $1,\dots,n$.  Asterisks indicate that a subscript or superscript runs over a range of values, determined by context.  For example, $F_\ast$ is often used to represent $F_k, k=0,1,\dots$.

Computations will be done in vector spaces over a fixed field $\FF$.  Computations in topological data analysis are typically done over finite fields, for example $\FF = \FF_2 = \ZZ/2\ZZ$, or the rationals $\QQ$, because homology requires exact computation of kernels and images.  Floating point arithmetic is typically avoided due to numerical issues.  One exception is Hodge theory, which uses the more familiar fields $\RR$ or $\CC$ -- for a numerical and application focused introduction see \cite{limHodge}.

\paragraph{Triangular Matrices}
A matrix $L$ is lower triangular if $L[i,j] = 0$ for $i < j$.  A matrix $U$ is upper triangular if $U[i,j] = 0$ for $i > j$.  We will use the symbols $L$ and $U$ to denote general lower and upper-triangular matrices respectively.  The inverse of a triangular matrix can be applied to a vector or matrix in the same time as an ordinary multiplication using using either the forward or backward substitution algorithm \cite{GVL}, which offers a distinct advantage compared to arbitrary matrices.

\paragraph{Pivot Matrices}
\begin{definition}\label{def:pivot_matrix}
A pivot matrix is a matrix in which every row and column has at most one non-zero element.
\end{definition}
In the context of matrix factorizations used in this paper, the non-zero element will always be 1 (the multiplicative identity of the field $\FF$) by convention.  The term pivot matrix refers to its use in recording pivots (last nonzeros of rows or columns) when computing matrix factorizations.
Pivot matrices are similar to permutation matrices in the sense that they map a single basis element to a single basis element, but contain the possibility that some rows and columns can be entirely zero so they are not generally invertible.

\paragraph{Echelon Pivot Matrices}
Echelon pivot matrices are pivot matrices with added structure. There are 4 types we consider
\begin{align}
  E_L &= \squareEL \;\;\;\; E_U = \squareEU\\
 \hat{E}_L &= \squareELh \;\;\;\;  \hat{E}_U = \squareEUh 
\end{align}
$E_L$ and $\hat{E}_U$ contain the pivots for variants of the column echelon form of a matrix, and $E_U$ and $\hat{E}_L$ contain the pivots for variants of the row echelon form of a matrix.  The $L$ and $U$ subscript indicates whether the matrix is lower or upper triangular.

\begin{definition}\label{def:el}
A matrix has the $E_L$ shape if it is the sum of rank 1 matrices created from basis vectors
\[ E_L = \sum_{(i,j)\in S} e_i e_j^T \]
where the set $S\subset \{1,\dots,m\} \times \{1,\dots,n\}$ contains the locations of the pivots. Since $E_L$ is a pivot matrix, for every $j$, there must be a unique $i$, therefore the pairs can be written as $(i(j),j)$. The function $i(j)$ is defined on the subset of columns that have a pivot, and must satisfy the following properties
\begin{enumerate}
\item $j_1<j_2 \implies i(j_1) < i(j_2)$ on the domain of $i(j)$
\item For every $j_1,j_2$ s.t. $j_1<j_2 \;\text{and}\; A[j_1] = 0 \implies A[j_2] =0$
\end{enumerate}
\end{definition}

The other echelon pivot matrices can be defined in terms of the $E_L$ shape.

\begin{definition}
 A matrix is of shape $E_U$ if its transpose is of shape $E_L$
 \[ (E_U)^T = E_L\]
\end{definition}

\begin{definition}
 A matrix is of shape $\hat{E}_L$ if its J-Conjugate is of shape $E_U$
 \[ J\hat{E}_LJ = E_U\]
\end{definition}

\begin{definition}
 A matrix is of shape $\hat{E}_U$ if its J-Conjugate is of shape $E_L$
 \[ J\hat{E}_UJ = E_L\]
\end{definition}

\paragraph{The J Matrix:} We will use $J$ to represent a linear operator that reverses row or column order:
\begin{definition}
$J$ is a square $n \times n$ matrix, such that
\[J[i,j] =
\begin{cases}
    1, & \text{if } i = n+1 - j\\
    0,              & \text{otherwise}
\end{cases}
\]
\end{definition}
In other words, it is the anti-diagonal permutation matrix.
Specifically when multiplied on the left, it reverses the row order. Similarly, it reverses the column order when multiplied on the right. It is its own transpose and inverse. We will often conjugate with the $J$ matrix, which reverses both row and column order and thus produces a reflection across the anti-diagonal. Note that this operation is distinct from taking the transpose of a matrix and cannot be expressed in terms of it.

The following are useful commutation relations between the J matrix and other matrix shapes

\begin{align}
  JL &= UJ \;\;\;\; &\squareDa\; \squareL = \squareU\; \squareDa \\
  JU &= LJ \;\;\;\; &\squareDa\; \squareU = \squareL\; \squareDa \\
  JE_L &= \hat{E}_UJ \;\;\;\; &\squareDa\; \squareEL = \squareEUh\; \squareDa \\
  JE_U &= \hat{E}_LJ \;\;\;\; &\squareDa\; \squareEU = \squareELh\; \squareDa
\end{align}

\subsection{Computing Homology of Spaces}\label{sec:computing_hom}
We will consider topological spaces encoded as cell complexes with a finite number of cells.  A {\em cell complex}, or CW complex $\cal X$ can be built inductively by starting with a discrete set of points (0-cells) $\calX^0$ called the $0$-skeleton, and inductively forming the $k$-skeleton $\calX^k$ from $\calX^{k-1}$ by adding open $k$-dimensional balls along their boundary to $\calX^{k-1}$ \cite{HatcherAT}.  Simplicial and cubical complexes are special cases of cell complexes which are often encountered in topological data analysis which offer a simplified interface for specifying cells and attaching maps.

From a cell complex $\cal X$, we apply the cellular chain functor to obtain a chain complex $C_\ast(\cal X)$.  We will consider this chain complex over a field $\FF$ instead of a more general ring, which will later allow us to use quiver representations.  The chain complex $C_\ast(\cal X)$ is a sequence of vector spaces $\{C_k(\cal X)\}$, $k=0,1,\dots$, with {\em boundary maps} $\partial_k: C_k \to C_{k-1}$, and $\partial_0 = 0$, with the property that $\partial_{k} \circ \partial_{k+1} = 0$.  

From a chain complex $C_\ast(\cal X)$, we can compute homology $H_k(\calX) = \ker \partial_{k} / \img \partial_{k+1}$.  This is a quotient vector space whose dimension encodes the number of $k$-dimensional ``holes'' in $\cal X$.  In practice, we can compute homology by computing matrix factorizations using the reduction algorithm of Zomorodian and Carlsson \cite{ZCComputingPH2005}
\begin{equation}
    \partial_k = R_k U_k
\end{equation}
where $R_k$ has unique column pivots (largest index with a non-zero entry) and $U_k$ is upper-triangular.  A basis for homology can be obtained by finding indices of zeroed-out columns of $R_k$ which do not appear as pivots in $R_{k+1}$, and taking the corresponding columns of $U_k$ as representatives of the coset in the quotient vector space $H_k(\cal X)$.  We will encode this idea of using a change of basis $U_k$ to reveal homology in the following definition:
\begin{definition}\label{def:hom_revealing_basis}
A {\em homology revealing} basis for $C_k$ is a pair $(B_k, \calI_k)$, where $B_k$ is a basis for $C_k$, and $\calI_k$ is an index set such that $\{b_i \in B_k\}_{i\in \calI_k} \subseteq B_k$ generates a basis for $H_k(C_\ast)$.  Explicitly, a basis for $H_k$ is
$$\{ [b_i] \mid b_i \in B_k, i\in \calI_k \}$$
\end{definition}
Note that a homology revealing basis is generally not unique.  Once we have chosen a basis $B_k$ and a set $\calI_k$, we will say a homology representative $x\in [x]$ is the preferred representative of $[x]$ if $x$ is written as linear combination of cycles exclusively in the set $\calI_k$.
From the reduction algorithm, we obtain a homology revealing basis $(U_k, \calI_k)$, where $\calI_k$ indexes the zeroed out columns of $R_k$ which do not appear as pivots in $R_{k+1}$.

When we consider a filtered cell complex $\calX_{t_0} \subseteq \calX_{t_1} \subseteq \dots$, we can put the rows and columns of $\partial_k$ in filtration order.  In this special case of filtrations, persistent homology can be computed using the reduction algorithm by adding an additional level of interpretation \cite{ZCComputingPH2005}.  Now every zeroed-out column of $R_k$ corresponds to a homology class which is born when the cell whose boundary is encoded in the corresponding column of $\partial_k$ is inserted.  This same class dies when a cell is inserted in $\partial_{k+1}$ which produces a column of $R_{k+1}$ with a pivot index matching the column index which produced homology.

\subsection{Computing Induced Maps}\label{sec:induced_maps}

Homology is a functor, meaning that topological maps $f: \calX \to \calY$ have associated linear maps $H_k(f):H_k(\calX) \to H_k(\calY)$.  We will consider maps $f$ which are cellular (meaning cells of $\calX$ map to sub-cell complexes of $\calY$), so can use the cellular chain functor to obtain a chain map $F_\ast : C_\ast(\calX) \to C_\ast(\calY)$.  In this section, we present an algorithm to compute an induced map on homology $\tilde{F}_k = H_k(f) : H_k(\calX) \to H_k(\calY)$ which can be used after running the reduction algorithm on $C_\ast(\calX)$ and $C_\ast(\calY)$ to obtain the induced map on homology in terms of the homology revealing bases.  While the reduction algorithm and its variants are used extensively for computing persistent homology, computing arbitrary induced maps is not very common in topological data analysis, and to the best of our knowledge \cref{alg:hom_ind_map} is the first account of how one may use the output of the reduction algorithm to compute arbitrary induced maps.

We consider a chain map $F_\ast : C_\ast \to D_\ast$, and homology revealing bases $(U_k^C, \calI_k^C)$, $(U_k^D, \calI_k^D)$.  The purpose of our algorithm is to write the image of a homology representative $[x]\in H_k(C_\ast)$, $[F_k x]\in H_k(D_\ast)$ in terms of its preferred representative.  This will allow us to read off coefficients to represent the induced map on homology $\tilde{F}_k: H_k(C_\ast) \to H_k(D_\ast)$ as a matrix.

\begin{algorithm}[htb]
\caption{Computation of induced map on homology.}
\label{alg:hom_ind_map}
\begin{algorithmic}[1]
\STATE{{\bf Input:} Homology representative $x = U_k^C[i]$ in $H_k(C_\ast)$,  $U_k^D, R_{k+1}^D$, from reduction algorithm applied to $\partial_\ast^D$, with index set $\calI^D_k$. Chain map $F_k$ in original basis.}
\STATE{{\bf Result:} Induced map on homology, $\tilde{F}_k[x]$}
\STATE{$y \gets (U_k^D)^{-1}F_k x$}
\STATE{$n \gets \dim D_k$}
\STATE{$\hat{\partial}_{k+1}^D \gets (U_k^D)^{-1} R_{k+1}$}
\FOR{$j = n,n-1,\dots,1$}
    \IF{$y[j] \ne 0$ and $j$ is a pivot of column $i$ of $\hat{\partial}_{k+1}^D$}
        \STATE{$\alpha \gets y[j] / \hat{\partial}_{k+1}^D[j,i]$}
        \STATE{$y \gets y - \alpha \hat{\partial}_{k+1}^D[i]$}
    \ENDIF
\ENDFOR
\RETURN{$y[\calI^D_k]$}
\end{algorithmic}
\end{algorithm}

\begin{proposition}
The homology class of $y$ in \cref{alg:hom_ind_map} is invariant.
\end{proposition}
\begin{proof}
Only boundaries are added to $y$, as columns of $\hat{\partial}^D_{k+1}$, so the homology class is invariant.
\end{proof}

\begin{proposition}
In \cref{alg:hom_ind_map}, at the end of the for-loop, $y$ will be the preferred representative for its homology class with respect to the basis $(U_k^D, \calI_k^D)$. 
\end{proposition}
\begin{proof}
This means that $y$ will only have non-zeros in indices that are in the index set $\calI_k^D$.

First, note that $\calI^D_k$ was constructed to be the indices of cycles in $D_k$ that did not appear in pivots of $R_k$, and $\hat{\partial}^D_k$ has the same pivots as $R_k$.  Second, note that because $x$ is a cycle, and $F_k$ is a chain map, $y$ must be a linear combination of cycles in the basis given by $U^D_k$.  Finally, the for-loop removes non-zero coefficients for any cycle that has a non-zero pivot in $\hat{\partial}^D_k$.  Thus, $y$ can only have non-zero coefficients for the cycles indexed by $\calI_k^D$.
\end{proof}

As a result, the coefficients returned by \cref{alg:hom_ind_map} will be the coefficients for the induced map on homology in terms of the homology revealing basis.  We can construct a full matrix representing $\tilde{F}_k$ in the bases generated by the homology revealing bases by applying this procedure for every preferred basis element in $U^C_k$ given by the index set $\calI^C_k$.

The reason why the induced maps never need to be explicitly computed in the reduction algorithm applied to filtered cell complexes as in \cite{ZCComputingPH2005} can be seen in the following proposition
\begin{proposition}\label{prop:filtration_preferred_basis}
Let $(U_k^i, \calI_k^i)$ and $(U_k^j, \calI_k^j)$ be homology revealing bases for a filtration $X_i \subseteq X_j$ computed using the ordered cell basis and the reduction algorithm.  Then the induced map on homology $H_k(X_i) \to H_k(X_j)$ from inclusion either (a) sends a basis element $[x] \in H_k(X_i)$ to $[0]$, or (b) sends a basis element $[x]\in H_k(X_i)$ to exactly one other basis element $[x'] \in H_k(X_j)$.  Furthermore, in case (b), the chain map sends the preferred representative of $[x]$ to the preferred representative of $[x']$.
\end{proposition}
\begin{proof}
Let $x$ be a preferred representative for $[x]\in H_k(\calX_i)$.  Note that because the reduction algorithm will produce the same result for the first $n_i$ columns of $\partial_k$, the inclusion map $C_k(\calX_i) \to C_k(\calX_j)$ has the form
$$F_k = \begin{bmatrix}
I\\ 0
\end{bmatrix}
$$
where $I$ denotes an identity and $(U^j_k)^{-1} F_k U^i_k = F_k$ takes the exact same form.  We know that because $x$ is a preferred representative, that it is a column of $U_k^i$, and the block identity in the chain map will map $x$ to the corresponding column in $U_k^j$.  Either that column is in $\calI_k^j$, in which case it is a preferred representative for a basis element of $H_k(\calX_j)$, or there must be a pivot with the corresponding index in $R_{k+1}^j$ indicating that the column is a boundary, so is in $[0]$.
\end{proof}
\begin{corollary}
The induced map on homology $\tilde{F}_k:H_k(\calX_i) \to H_k(\calX_j)$ is in $E_U$ form.
\end{corollary}
\begin{proof}
This follows from using the ordering inherited on the basis for homology from the ordering on the basis for the cell complex.
\end{proof}

\subsection{Quiver Representations}\label{sec:quivertda}

The application of quiver representations to persistent and zigzag homology has been of interest ever since it was introduced in the context of zigzag barcodes \cite{ZZtheory2010}.  A fairly recent survey of existing results and applications to topological data analysis can be found in a monograph by Oudot \cite{Oudot}.

A quiver is a mathematical term for what is known in computer science as a directed multi-graph, which is a collection of vertices (nodes) $V$ and a multiset $E$ of edges in $V\times V$.  We will only consider directed multi-graphs consisting of a finite number of nodes and edges.
A quiver representation is a directed multi-graph $\calQ(V, E)$ where every vertex $v_i\in V$ has an associated vector space $V_i$ over a common field $\FF$, and each directed edge $(v_i, v_j) \in E$ has an associated $\FF$-linear transformation $A_{i,j}:V_i \to V_j$.  Two quiver representations $\calQ_1(V^1, E^1), \calQ_2(V^2, E^2)$ are said to be isomorphic if the underlying graphs are isomorphic and there are change of bases $B_i$ for each $V^2_i$ so that $A^1_{i,j} = (B_i)^{-1} A^2_{i,j} B_j$.  Quiver representation theory is concerned with classifying quiver representations up to isomorphism, a problem that originated with classification of Lie Algebras \cite{derksenQuiverRepresentations2005}.

The ability to classify arbitrary quiver representations relies entirely on the underlying undirected graph.  Both persistent and zigzag homology are quiver representations of type $A_n$, for which the underlying graph is a line graph on $n$ vertices, where $n$ can be any finite positive integer. A theorem due to Gabriel shows that the collection of underlying graphs of quiver representations that have a finite number of indecomposable representations are known as the Dynkin diagrams, which include $A_n$, as well as several other classes of graph \cite{gabrielI}.  We will refer to type $A_n$ quivers where all arrows point in the same direction {\em persistence-type} quivers, and type $A_n$ quivers where arrows alternate direction {\em zigzag-type} quivers.

Quiver representations arise naturally from diagrams of topological spaces through the homology functor which associates topological vector spaces $X$ with vector spaces $H_k(X)$, and maps $f:X\to Y$ to linear transformations $\tilde{F}_k:H_k(X)\to H_k(Y)$.  This means that the homology functor turns diagrams of topological spaces into diagrams of vector spaces (quiver representations). Note that neither the diagram of topological spaces nor the diagram of vector spaces is required to commute.

\begin{example}
Persistent homology studies diagrams of spaces
$$
\begin{tikzcd}
X_0\ar[r, "f_0"] & X_1\ar[r,"f_1"] &\dots
\end{tikzcd}
$$
which produces a quiver representation
$$
\begin{tikzcd}
H_k(X_0)\ar[r, "(F_0)_k"] &H_k(X_1)\ar[r, "(F_1)_k"]& \cdots
\end{tikzcd}
$$
which corresponds to a Dynkin diagram of type $A$.
\end{example}

The advantage of using quiver representations to study persistent and zigzag homology as opposed to approaches that operate directly on chain complexes is that the application of the homology functor is {\em embarassingly parallelizable}.  In fact, this can be applied to general diagrams of spaces.

\begin{algorithm}
\caption{Obtain a quiver representation from a diagram of spaces}
\label{alg:top_to_quiv}
\begin{algorithmic}[1]
\STATE{{\bf Input:} Directed multi-graph $(V,E)$ with associated spaces $X_i$ and maps $f_{i,j}$}
\STATE{{\bf Result:} Directed multi-graph $(V,E)$ with associated vector spaces $H_k(X_i)$ and maps $\tilde{F}_{i,j} = H_k(f_{i,j})$}
\FOR{$i\in V$}
    \STATE{Obtain chain complex $C_\ast(X_i)$}
    \STATE{Obtain homology revealing bases $(U^i_\ast, \calI^i_\ast)$ as well as reduced boundaries $R^i_\ast$ using reduction algorithm \cite{ZCComputingPH2005}.}
\ENDFOR
\FOR{$(i,j)\in E$}
    \STATE{obtain $(\tilde{F}_{i,j})_\ast = H_\ast(f_{i,j})$ using \cref{alg:hom_ind_map}.}
\ENDFOR
\end{algorithmic}
\end{algorithm}

The important observation is that both for loops of \cref{alg:top_to_quiv} have {\em completely independent iterations}.  That is, computing homology of $X_i$ can be done completely independently of computing homology for $X_j$.  Similarly, computing the induced maps $\tilde{F}_{i,j}$ only requires the pre-computed information for the source and target of the associated edge, and is independent of the computation on any other edge.  This means the algorithm is {\em embarassingly parallelizable}, meaning that each for-loop iteration can be computed in parallel given enough processors.
Furthermore, this is true for any diagram of spaces, meaning it applies not only to the persistent and zigzag homology diagrams that we will study in this paper, but also other situations such as multiparameter persistence \cite{carlssonTheoryMultidimensionalPersistence2009}.

\subsection{Type A Quiver Representations}\label{sec:type_a_quiver}
We will now focus on classification of type $A_n$ quiver representations, which appear for both persistent and zigzag homology.  The indecomposable representations of these quiver representations are known as interval indecomposables \cite{gabrielI, ZZtheory2010, Oudot}, and have the form
\begin{center}
\begin{tikzcd}
I[b,d] = \cdots \arrow[r, dash] &0 \arrow[r, dash]  &\FF \arrow[r, dash]  &\cdots \arrow[r, dash] &\FF \arrow[r, dash] & 0 \arrow[r, dash] & \cdots
\end{tikzcd}
\end{center}
where $b$ denotes the first index at which a copy of $\FF$ appears, and $d$ denotes the final index where $\FF$ appears, all vector spaces with index $i\in [b,d]$ also have a copy of $\FF$, with identity maps along all edges connecting two copies of $\FF$ and zero maps along all other edges. In other words, any quiver representation of type $A_n$ is isomorphic to the direct sum of these indecomposables
$$\calQ \cong \bigoplus_i I[b_i,d_i]$$
As a convention, we will use the lexicographical (total) order on $\ZZ^2$ when ordering interval indecomposables, using the parameters $b, d$.  When the quiver representation is produced from induced maps on homology, the multiset $\{(b_i, d_i)\}$ is the barcode.

\begin{definition}
The {\em companion matrix} of a quiver representation $\calQ$ is the block matrix which has non-zero blocks in the non-zero entries of the adjacency matrix of the underlying directed graph, where blocks are filled by the linear transformations along the corresponding edges.  By necessity, the size of the $i$-th block must be the dimension of $V_i$ in the quiver representation.
\end{definition}
These matrices act on the vector space $V = \bigoplus_i V_i$ by sending vectors to their images in each linear transformation in $Q$.  While general quiver representations may have multiple arrows between vector spaces, companion matrices can only represent those which have a underlying graph with at most one edge for each source-target pair -- this will not limit our study of type $A_n$ quiver representations as they satisfy this property.

For example, a persistence quiver
\begin{tikzcd}[sep=small] P_4 = \cdot &\cdot \arrow[l] &\cdot \arrow[l] &\cdot  \arrow[l] \end{tikzcd} will have a companion matrix of the form
$$\begin{bmatrix}
0 & A_1\\
& 0& A_2\\
&& 0 & A_4\\
&&& 0
\end{bmatrix}$$
whereas a zigzag quiver
\begin{tikzcd}[sep=small] Z_4 = \cdot \arrow[r] &\cdot &\cdot \arrow[l, rightarrow] \arrow[r] &\cdot \end{tikzcd}
will have a companion matrix of the form
$$\begin{bmatrix}
0\\
A_1 & 0 & A_2\\
& & 0\\
&& A_3 & 0
\end{bmatrix}$$

Quiver representation isomorphism classes are maintained by conjugation of the companion matrix by block-diagonal change of bases matrices. For example two persistence quivers of type \begin{tikzcd}[sep=small]P_3= \cdot &\cdot \arrow[l] &\cdot \arrow[l] \end{tikzcd} with companion matrices $A$ and $B$ respectively are isomorphic if there exists an invertible matrix $M= M_1 \oplus M_2 \oplus M_3$ acting on $V$ such that
$$\begin{bmatrix}
0 & A_1\\
& 0& A_2\\
&& 0
\end{bmatrix}
= \begin{bmatrix}
M_1\\
& M_2\\
&& M_3
\end{bmatrix}
\begin{bmatrix}
0 & B_1\\
& 0& B_2\\
&& 0
\end{bmatrix}
\begin{bmatrix}
M_1^{-1}\\
& M_2^{-1}\\
&& M_3^{-1}
\end{bmatrix}
$$
From this, it is clear that two quiver representations can not be isomorphic if they do not share the same underlying directed multi-graph, and the dimensions of the vector spaces are not identical.

An {\em indecomposable factorization} of the companion matrix $A$ is a factorization $A = B T B^{-1}$, where $B$ is an invertible (change of basis) matrix, and $T = \bigoplus_i I[b_i, d_i]$ is the matrix of indecomposables. For example, in the case where $P_4$ and $Z_4$ both have indecomposable matrices $T =  I[1,1] \oplus \red{I[1,4]} \oplus [2,3]$, the corresponding matrices are
\begin{equation}\label{eq:ind_fact}
T_{P_4} = \begin{bmatrix}
0\\
& \red{0} & \red{1} & \red{0} & \red{0}\\
& \red{0} & \red{0} & \red{1} & \red{0}\\
& \red{0} & \red{0} & \red{0} & \red{1}\\
& \red{0} & \red{0} & \red{0} & \red{0}\\
&&&&& 0 & 1\\
&&&&& 0 & 0
\end{bmatrix}\qquad
T_{Z_4} = \begin{bmatrix}
0\\
& \red{0} & \red{0} & \red{0} & \red{0}\\
& \red{1} & \red{0} & \red{1} & \red{0}\\
& \red{0} & \red{0} & \red{0} & \red{0}\\
& \red{0} & \red{0} & \red{1} & \red{0}\\
&&&&& 0 & 1\\
&&&&& 0 & 0
\end{bmatrix}
\end{equation}
where the information for the indecomposable $I[1,4]$ is colored in \red{red}.  Note the indecomposable blocks all appear as adjacency matrices of sub-graphs of the underlying directed graph of the quiver.  This means that even though the indecomposables are written with the same symbolic notation, the indecomposable matrices are not identical due to the different directions of arrows.

A {\em barcode factorization} of the companion matrix $A$ is a factorization $A = B\Lambda B^{-1}$, where $B$ is a {\em block-diagonal} invertible matrix (representing a quiver representation isomorphism), where the block sizes are compatible with dimensions of the associated vector spaces, and 
$$\Lambda = P T P^T$$
is the barcode matrix, where $P$ is a permutation that maintains the block structure of $A$. Alternatively, we'll say $\Lambda$ is the {\em barcode form} 
of the companion matrix $A$, or simply the barcode form of the quiver representation.  Continuing the previous example, in the case where $P_4$ and $Z_4$ both have barcode matrices $\Lambda \cong  I[1,1] \oplus \red{I[1,4]} \oplus [2,3]$, the corresponding matrix representations are
\begin{equation}\label{eq:bar_fact}
\Lambda_{P_4} = \begin{bmatrix}
&& 0 & 0\\
&& \red{1} & 0\\
&&&& \red{1} & 0\\
&&&& 0 & 1\\
&&&&&& \red{1}\\
&&&&&& 0\\
\; &\;
\end{bmatrix} \qquad
\Lambda_{Z_4} = 
\begin{bmatrix}
\\
\\
0 & \red{1} &\; &\; & \red{1} & 0\\
0 & 0 &\; &\; & 0 & 1\\
\\
\\
&&&& \red{1} & 0 &\;
\end{bmatrix}
\end{equation}
The information for the indecomposable $I[1,4]$ is colored in \red{red}.
Note that because the underlying graphs are different the matrices $\Lambda$ are not equal even though the notation for the interval decomposition is superficially the same.  In both quiver representations, the ranks of the vector spaces are $2,2,2,1$.  The advantage of the barcode factorization is that $B$ clearly represents a quiver representation isomorphism due to its block structure.

Extracting the intervals $I[a,b]$ from a barcode factorization requires tracing the image of maps through the quiver, and the advantage compared to the indecomposable decomposition is that the starting point of an interval is clear.  We extract the intervals from $\Lambda$ by sweeping through the blocks left-to-right, as seen in \cref{alg:extraction}.
\begin{algorithm} 
\caption{Barcode Extraction}
\label{alg:extraction}
\begin{algorithmic}[1]
\STATE{{\bf Input:} Barcode matrix $\Lambda$, dimensions of vector spaces $V_i$ and directions of arrows in quiver.}
\STATE{{\bf Result:} Barcode $\calB$}
\FOR{$i = 1,\dots, n$}
\FOR{$j = 1,\dots, \rank V_i$}
\IF{$V_{i-1} \to V_i$}

\IF{Row $j$ in block $i$ contains a non-zero in column $j'$ of block $i-1$}
\STATE{Extend the bar at index $j'$ of block $i-1$ to have index $j$ in block $i$.}
\ELSE
\STATE{Begin a bar with index $j$ in block $i$}
\ENDIF

\ELSIF{$V_{i-1}\leftarrow V_i$}

\IF{Column $j$ in block $i$ contains a non-zero in row $j'$ of block $i-1$}
\STATE{Extend the bar at index $j'$ of block $i-1$ to have index $j$ in block $i$.}
\ELSE
\STATE{Begin a bar with index $j$ in block $i$}
\ENDIF
\ELSE
\STATE{($i = 1$)}
\STATE{Begin bar with index $j$ in block 1}
\ENDIF
\ENDFOR
\ENDFOR
\RETURN{$\calB$}
\end{algorithmic}
\end{algorithm}
If we keep track of the indices used in each extension of a bar, we have the information necessary to form the permutation $P$ so that $P \Lambda P^T = T$.  
\begin{proposition}\label{prop:barcode_form}
A companion matrix is in barcode form if and only if its blocks are pivot matrices.
\end{proposition}
\begin{proof}
If all blocks of a companion matrix are pivot matrices, then every iteration of the for-loop in \cref{alg:extraction} will find at most one index that can be used to extend a bar.  The set of bars found by the algorithm gives the indecomposables, so the matrix is in barcode form.

If the matrix is in barcode form, we'll consider the representation of the map $A_{i}: V \to W$, where either $V = V_i, W = V_{i+1}$ or $V = V_{i+1}, W = V_i$.  Because the matrix is in barcode form, every basis vector $v\in V$ maps to either exactly one basis vector of $W$ (continuing a bar), in which case the corresponding column of $A_i$ has exactly one non-zero, or maps to zero (the bar ends at $V$), in which case the corresponding column of $A_i$ is zero.  Every basis vector $w\in W$ is either in the image of a basis vector of $V$, in which case the corresponding row of $A_i$ has exactly one non-zero, or is the start of a new bar, in which case the row of $A_i$ is zero.  Thus $A_i$ is a pivot matrix because it has at most one non-zero in each row and column.
\end{proof}

For persistence-type quivers, there is a relationship between the indecomposable factorization and the Jordan normal form of a matrix, observed in equation \cref{eq:ind_fact}.  Because interval indecomposables have the structure of a sub-graph of the directed graph underlying the quiver, in the case of persistence quivers they will all have the form of a Jordan zero block.  While the Jordan form does not generally exist for non-algebraically closed fields, Greg Henselman showed in his Ph.D. dissertation that nilpotent operators can always be decomposed into Jordan-zero blocks \cite{henselman2017}, and applied this along with the use of Schur complements to computing persistent homology of nested chain complexes.  We also note that because homology of spaces starts with integer chain complexes, that we only need to use the field of fractions $\QQ$ for computation so $H_\ast(~\cdot~; \RR) \cong H_\ast(~\cdot~; \QQ)$.  In practice, we avoid any issues with finite precision observed when computing the Jordan form with real coefficients \cite{golubIllConditionedEigensystemsComputation1976} through use of rational numbers.

\subsection{Generic Quiver Computations}
There is a correspondence between diagrams encoding quiver representations and block matrices encoding their companion matrices, and certain operations are easier to express using one notation or the other.  In this section we establish some lemmas that apply to any quiver representation.

\begin{lemma}\label{lem:diagram_matrix}
A change of basis (quiver representation isomorphism) at a single vertex via an invertible matrix $M$ can be represented as
\begin{center}
\begin{tikzcd}[column sep=large]
\cdot  \ar[dr,leftarrow, sloped, "A_0", pos=0.5] & & \cdot && \cdot \ar[dr,leftarrow, sloped, "A_0 M", pos=0.5] & & \cdot  \\
\vdots & \cdot \ar[ur,leftarrow, sloped, "B_0", pos=0.5] \ar[dr,leftarrow, sloped, "B_m", pos=0.5] & \vdots &\cong& \vdots & \cdot \ar[ur,leftarrow, sloped, "M^{-1} B_0", pos=0.5] \ar[dr,leftarrow, sloped, "M^{-1} B_m", pos=0.5]  & \vdots \\
\cdot \ar[ur,leftarrow, sloped, "A_n", pos=0.5]  & & \cdot && \cdot \ar[ur,leftarrow, sloped, "A_n M", pos=0.5] & & \cdot 
\end{tikzcd}
\end{center}
\end{lemma}
\begin{proof}
This follows immediately from a change of basis on the central vector space in the diagram.
\end{proof}
If the quiver representation is representable by a companion matrix, this diagramatically encodes the isomorphism represented by conjugation with the block change of basis matrix
$$ \begin{bmatrix}
I\\
&M\\
&& I\\
\end{bmatrix}
$$
where $M$ acts on the central vector space. We see that this only affects linear transformations that have the center vertex as a source or target. For any vector spaces that do not have an arrow to or from the center vector space are multiplied by an identity on both left and right and are unaffected.

\cref{lem:diagram_matrix} implies the following two corollaries which set forth the rules for our {\em matrix passing algorithms}.
\begin{corollary}\label{cor:targ_iso}
Passing an invertible matrix $M$ through a target yields
\begin{center}
     \begin{tikzcd}[row sep=small, column sep=large]
    \cdot \ar[r, leftarrow, "A"] &\cdot\ar[r, leftarrow, "M B"] & \cdot 
    & \cong &
    \cdot \ar[r, leftarrow, "AM"] &\cdot\ar[r, leftarrow, "B"] & \cdot 
    \end{tikzcd}
\end{center}
\end{corollary}

\begin{corollary}\label{cor:source_iso}
Passing an invertible matrix $M$ through a source yields
\begin{center}
     \begin{tikzcd}[row sep=small, column sep=large]
    \cdot \ar[r, leftarrow, "AM"] &\cdot\ar[r, leftarrow, "B"] & \cdot 
    & \cong &
    \cdot \ar[r, leftarrow, "A"] &\cdot\ar[r, leftarrow, "MB"] & \cdot 
    \end{tikzcd}
\end{center}
\end{corollary}

Notice that we draw arrows from right to left in the above diagrams.  This is simply because the matrix $M$ is closest to the vertex undergoing a change of basis.

\section{Algorithms for Canonical Forms of Type A Quiver Representations}\label{sec:quiveralg}
In this section, we will describe our algorithm for computing the barcode form of a type A quiver representation.  As we will see, the details of the algorithm depends on the direction of the arrows in the quiver, but there are core components which are the same.

There are two basic linear algebra operations we need as primitives for the algorithm.
\begin{enumerate}
\item Triangular factorizations
\item Shape commutation results with E-type matrices
\end{enumerate}
We will first discuss these. Next we will consider the algorithm when all the the arrows point in the same direction. We will then show how this algorithm can be modified to the case of alternating arrow directions. Finally we will specify the full general algorithm.

\subsection{Triangular Factorizations}
In this section, we will discuss computing decompositions of a matrix $A$ into triangular matrices with row or column pivoting. Specifically we will start with the LEUP decomposition. Variants include PLEU, UELP and PUEL form, which can be derived using the LEUP decomposition of either transposed or J-Conjugated versions of $A$.  These factorizations are all variants of the standard LU decomposition \cite{GVL}, but we will need these specific forms for our algorithm.

Given a matrix $A$, we will describe an algorithm that will maintain a block invariant at each step $i$
\begin{equation} \label{eq:tri_invariant}
A = 
\begin{bmatrix} L_{11} &\\ L_{21} & I \end{bmatrix}
\begin{bmatrix} E_{11} & \\ & \tilde{A} \end{bmatrix}
\begin{bmatrix} U_{11} & U_{12}\\ & I \end{bmatrix}
P
\end{equation}

\begin{algorithm} 
\caption{LEUP factorization}
\label{alg:LEUP}
\begin{algorithmic}[1]
\STATE{{\bf Input:} $m\times n$ matrix $A$}
\STATE{{\bf Result:} Factorization $A = LEUP$}
\STATE{$L \gets I_m$, $E \gets A$, $U\gets I_n$, $P\gets I_n$}
\STATE{$i = 1, j = 1$}
\WHILE{$i <= m ~\&~ j <= n$} 
    \IF{row $i$ has a non-zero in column $j' \ge j$ of $E$}
        \STATE{Swap columns $j$, $j'$ in $E$ and $U$, and rows $j$, $j'$ of $U$ and $P$}
        \STATE{Use \cref{eq:LEUP_schur} to form Schur complement with respect to $i,j$ entry in $E$ and update $L$ and $U$.}
        \STATE{$i \gets i + 1$, $j\gets j+1$}
    \ELSE
        \STATE{$i \gets i+1$}
    \ENDIF
\ENDWHILE
\end{algorithmic}
\end{algorithm}
\begin{proposition}
\cref{alg:LEUP} computes a factorization $A = LEUP$, where $E$ has an $E_L$ shape.
\end{proposition}
\begin{proof}
We will show that the invariant \cref{eq:tri_invariant} is maintained at each step of the algorithm. Note that the loop increments $i$ each iteration, so we use $i$ as our index.  For rows of $E$, as well as rows and columns of $L$, we will use the block index set $1 = \{0,\dots,i-1\}$, and the block index set $2 = \{i+1,\dots,m-1\}$, and for columns of $E$, as well as rows and columns of $U$ we will use the block index set $1 = \{0, \dots, j-1\}$ and block index set $2 = \{j+1,\dots,n-1\}$.

Assume the invariant is maintained at the beginning of iteration $i$.  In the case that there are no non-zero entries in $A_{ij}$ or $A_{i2}$, we simply increment $i$ (in the else clause of the while loop), and the invariant is trivially maintained.

In the case there is a non-zero entry in $A_{ij}$ or $A_{i2}$, note that swapping the $j$ and $j'$ columns of $E$ with the additional row and column swaps in $U$ and $P$ does not affect the invariant structure, as
\begin{equation}\label{eq:perm_tri_comm_block}
\begin{bmatrix}
I & \\
& P
\end{bmatrix}
\begin{bmatrix}
U_{11} & U_{12} \\
& I
\end{bmatrix} =
\begin{bmatrix}
U_{11} & U_{12}P^T\\
& I
\end{bmatrix}
\begin{bmatrix}
I & \\
& P
\end{bmatrix}
\end{equation}
The relation is easily verified by computing the products.  Now we will assume that the $i,j$ entry of the matrix $E$ is non-zero.  We can break up the matrices into
\begin{equation}\label{eq:LEUP_invariant}
A = 
\begin{bmatrix}L_{11} & &\\ L_{i1} & 1 &\\ L_{21} & & I\end{bmatrix}
\begin{bmatrix}E_{11} & &\\ & A_{ij} & A_{i2}\\ & A_{2j} & A_{22}\end{bmatrix}
\begin{bmatrix}U_{11} & U_{1j} & U_{22}\\ & 1 &\\& & I\end{bmatrix}
P
\end{equation}

We can write the interior matrix as
\begin{equation}\label{eq:LEUP_schur}
\begin{bmatrix}E_{11} & &\\ & A_{ij} & A_{i2}\\ & A_{2j} & A_{22}\end{bmatrix}
=
\begin{bmatrix}I & &\\ & 1 &\\ &A_{2j}A_{ij}^{-1} &I\end{bmatrix}
\begin{bmatrix}E_{11} & &\\ & A_{ij} & \\ &  & S\end{bmatrix}
\begin{bmatrix}I & &\\& 1 &A_{ij}^{-1}A_{i2}\\ & & I\end{bmatrix}
\end{equation}
where $S$ is the Schur complement $S = A_{22} - A_{2j} A_{ij}^{-1} A_{i2}$.  We can then pass off the left matrix to $L$ and the right matrix to $U$, and now group $A_{ij}$ with the echelon block $E_{11}$.  This maintains the invariant in \cref{eq:tri_invariant} for every iteration of the algorithm.  Note that because we may increment $i$ without incrementing $j$, that the matrix $E$ will be of type $E_L$.
\end{proof}

\begin{proposition}\label{prop:LEUP_time}
For an $m\times n$ matrix $A$, the factorization $A = LEUP$ computed by \cref{alg:LEUP} takes $O(mn \min(m,n))$ time.
\end{proposition}
\begin{proof}
The majority of the work in the algorithm is done when there is a non-zero pivot in $A_{ij} = A[i,j]$. 
To determine if there is a non-zero entry in row $i$ takes $O(n)$ time, since we may need to check every entry.  Checking every row thus takes $O(mn)$ time.
Swapping two columns can take $O(m)$ time if the columns are dense.  We may do up to $n$ column swaps, for another factor of $O(mn)$ time.

Forming the Schur complement in \cref{eq:LEUP_schur} requires forming a rank-1 outer product $A_{2j}A_{ij}^{-1}A_{i2}$ and performing a matrix subtraction.  The number of entries in these matrices is $O(mn)$, which bounds the time it takes to perform these operations.  This is computed up to $\min(m,n)$ times in the factorization, since we eliminate the other entries in each row and column when taking the Schur complement, ensuring we will never have to perform the same operation for a duplicate value of $i$ or $j$.

Finally, note that combining the lower and upper triangular terms in \cref{eq:LEUP_schur} with those in \cref{eq:LEUP_invariant} simply requires moving the non-zero blocks into the correct positions \cite{GVL}.  This takes $O(m)$ time for the lower triangular matrix, and $O(n)$ time for the upper triangular matrix.

The time to compute Schur complements dominates the computation.  This is done $O(\min(m,n))$ times, so the total time to form the factorization is $O(mn \min(m,n))$.
\end{proof}
Note that sparsity in the matrices, and the use of a sparse matrix implementation may potentially reduce the run time considerably.

\subsubsection{Other Triangular Factorizations}

Three other factorizations can be derived from the LEUP factorization using transposes and $J$-conjugation:
\begin{align}
    A &= (A^T)^T = (LE_LUP)^T = P^TU^T(E_L)^TL^T = \tilde{P}\tilde{L}E_U\tilde{U}\\
    A &= JJAJJ = JLE_LUPJ = \tilde{U}JE_LUJ \tilde{P} = \tilde{U} \hat{E}_U JJ\tilde{L} \tilde{P}= \tilde{U} \hat{E}_U  \tilde{L} \tilde{P}\\
    A &= JJAJJ = JPLE_UUJ = \tilde{P} JLE_UJ\tilde{L} = \tilde{P} \tilde{U} JJ \hat{E}_L \tilde{L} =  \tilde{P} \tilde{U} \hat{E}_L \tilde{L}
\end{align}

In practice, each of these factorizations can be computed through straightforward modifications to the loop and index orders in \cref{alg:LEUP} and can be computed in $O(mn\min(m,n))$ time as well.  This gives us four related factorizations to work with: $LE_LUP$, $PLE_U U$, $UE_U LP$, and $PUE_L L$.

\subsection{Shape Commutation Relations}\label{sec:shape_comm}

Now, we'll consider shape commutation relations of echelon matrices with triangular matrices. We will first show the following commutation relationship, and derive others from this.
\begin{proposition}\label{prop:EL_L_comm}
Given an echelon pivot matrix $E_L$ and lower triangular matrix $L$, we can rewrite the product $E_L L$ in the following way
\begin{equation}\label{eq:perm_tri_comm}
     E_L L = \tilde{L}E_L
\end{equation}
Where $\tilde{L}$ matrix is also a lower triangular matrix.
\label{prop:shape_comm}
\end{proposition}

\begin{proof}
We'll assume that $L$ has unit entries, as in \cref{alg:LEUP}. We'll use the construction
$$\tilde{L} = E_L (L - I) E_L^\dagger + I$$
where $E_L^\dagger$ is the pseudo-inverse, so when $E_L$ has unit entries, then $E_L^\dagger = E_L^T$.  One can easily verify that $\tilde{L}$ is lower-triangular.  We'll now verify that \cref{eq:perm_tri_comm} holds.  Forming the product $\tilde{L} E_L$, we have
\begin{align*}
\tilde{L} E_L &= E_L (L - I) E_L^\dagger E_L + E_L\\
&= E_L L E_L^\dagger E_L - E_L + E_L\\
&= E_L L E_L^\dagger E_L
\end{align*}
Let $k$ be the rank of $E_L$.  Then $E_L^\dagger E_L$ acts on the right as the projection onto the first $k$ columns of $L$.  Now, note that $E_L L$ only acts on $L$ by selecting the first $k$ rows of $L$ and mapping row $i$ to row $i(j)$ using the notation in \cref{def:el}.  Because $L$ is lower-triangular, projection onto the first $k$ columns does not affect the first $k$ rows, and so we have $E_L L = E_L L E_L^\dagger E_L$.

\end{proof}

\begin{proposition}\label{prop:other_comm}
The following shape commutation relations also hold
\begin{align}
  L\hat{E}_L &= \hat{E}_L\tilde{L} \;\;\;\; &\squareL \; \squareELh  = \squareELh \; \squareL \\
  UE_U &= E_U\tilde{U} \;\;\;\; &\squareU \; \squareEU  = \squareEU \; \squareU \\
  \hat{E}_UU &= \tilde{U}\hat{E}_U \;\;\;\; & \squareEUh \; \squareU = \squareU \; \squareEUh
\end{align}
\end{proposition}
\begin{proof}
Taking transpose on the commutation result for $E_L$
\[ (E_LL)^T = (\tilde{L}E_L)^T \]
\[ L^TE_L^T = E_L^T\tilde{L}^T \]
Rewriting to denote the shapes we get,
\[ UE_U = E_U\tilde{U} \]
Taking the J-Conjugate of the $E_L$ commutation result we have,
\[ J(E_LL)J = J(\tilde{L}E_L)J \]
\[ JE_LLJ = \hat{E}_UJLJ = \hat{E}_UU \]
\[ J\tilde{L}E_LJ = \tilde{U}JE_LJ =\tilde{U}\hat{E}_U \]
We get the commutation result for $\hat{E}_U$
\[ \hat{E}_UU = \tilde{U}\hat{E}_U \]
Taking the transpose of the above result, we get,
\[ (\hat{E}_UU)^T = (\tilde{U}\hat{E}_U)^T \]
\[ U^T\hat{E}_U^T = \hat{E}_U^T\tilde{U}^T \]
Rewriting to denote shapes,
\[ L\hat{E}_L = \hat{E}_L\tilde{L} \]
\end{proof}

\begin{proposition}\label{prop:time_comm}
If the $E$-type matrix is size $m\times n$, the shape commutations in \cref{prop:EL_L_comm} and \cref{prop:other_comm} take $O(\max(m,n)^2)$ time.
\end{proposition}
\begin{proof}
The work is done in simply forming the $\tilde{L}$ or $\tilde{U}$ matrices.  These are either $m\times m$ or $n\times n$ matrices, depending on the commutation being performed.  Note that the entries in this new matrix come directly from $L$ or $U$ respectively, using the $E$ matrix to shift indices appropriately, so no work must be done other than reading off the correct entries.
\end{proof}
More detailed bounds might be obtained, but the above covers all possibilities. Again, sparsity may be exploited to make this bound pessimistic.

\begin{proposition}\label{prop:comm_invert}
Let $\widetilde{L^{-1}}$ satisfy $L^{-1}\hat{E}_L = \hat{E}_L \widetilde{L^{-1}}$, and let $\tilde{L}$ satisfy $L \hat{E}_L = \hat{E_L} \tilde{L}$.  Then $\hat{E}_L \widetilde{L^{-1}} \tilde{L} = \hat{E_L}$
\end{proposition}
\begin{proof}
\begin{align*}
    \hat{E_L} &= L^{-1} L \hat{E}_L\\
    &= L^{-1} \hat{E}_L \tilde{L}\\
    &= \hat{E}_L \widetilde{L^{-1}}\tilde{L}
\end{align*}
\end{proof}
One may prove similar results for the other commutation types.  As a consequence, it does not matter if a triangular matrix is inverted before or after a shape commutation.

\subsection{Sequential Algorithm for Barcode Form}
We now have the components necessary to compute the barcode form of a type A quiver representation by changing bases to put each matrix in pivot form.  This accomplished in two passes, the first of which leaves each matrix as the product of $E_L L$ or $\hat{E}_L L$, and the second of which uses the shape commutation relations to remove the $L$ factors.  See \cref{alg:typeA_alg_ri} for details.

\begin{algorithm}[htb]
\caption{Obtain Barcode factorization of type $A$ quiver: Rightward-initial}
\label{alg:typeA_alg_ri}
\begin{algorithmic}[1]
\STATE{{\bf Input:} Type $A_n$ quiver representation.}
\STATE{{\bf Result:} Barcode form of quiver representation.}
\FOR{edge $i = 0,\dots,n-1$}
    \IF{$\leftarrow$}
        \STATE{Apply $LE_LUP$ factorization to matrix on edge $i$.}
        \STATE{Pass $UP$ factors to edge $i+1$.}
    \ELSE
        \STATE{Apply $PU\hat{E}_LL$ factorization to matrix on edge $i$.}
        \STATE{Pass $PU$ factors to edge $i+1$.}
    \ENDIF
\ENDFOR
\STATE{Pass $L$ term on edge $n$ to edge $n-1$.}
\FOR{edge $i=n-1,\dots, 0$}
    \IF{$\leftarrow$}
        \STATE{Commute $L_1E_LL_2 = \tilde{L}_1E_L$.}
        \STATE{Pass $\tilde{L_1}$ to edge $i-1$.}
    \ELSE
        \STATE{Commute $L_1\hat{E}_LL_2 = \hat{E}_L\tilde{L}_2$.}
        \STATE{Pass $\tilde{L_2}$ to edge $i-1$.}
    \ENDIF
\ENDFOR
\STATE{Change basis on first vector space to remove $L$ term on first edge.}
\end{algorithmic}
\end{algorithm}

We can also initiate the first sweep from the right to the left, to get a leftward initial algorithm.  In this case, we replace the $LE_L UP$ factorization with a $PLE_U U$ factorization and $PU\hat{E}_L L$ with $U\hat{E}_U LP$ and pass the $P$ and $L$ factors while working right to left.  Similarly, we replace the shape commutations with the $E_L$-type matrices with the appropriate commutations with $E_U$ type matrices.

Finally, note that the commutation relations established in \cref{sec:shape_comm} do not change the nonzero structure of the $E$ matrices.  Thus, it is possible to extract the barcode without performing the backward pass of \cref{alg:typeA_alg_ri} if one does not care to form the change of basis.

\subsubsection{Time Complexity}\label{sec:sequential_time_complexity}
We'll analyze the complexity of \cref{alg:typeA_alg_ri} for a quiver of type $A_{n+1}$ (with $n$ edeges).  For simplicity, let $d$ denote the largest dimension of a vector space in the quiver representation.
\begin{proposition}\label{prop:sequential_time_complexity}
The barcode form of a quiver representation of type $A_n$ with $d = \max\{V_0, \dots, V_n\}$ is computed by \cref{alg:typeA_alg_ri} in $O(nd^3)$ time.
\end{proposition}
\begin{proof}
In the forward pass, for each edge either a LEUP or PUEL factorization is formed in $O(d^3)$ time.  Then the U and P factors are applied to the next edge for an additional cost of $O(d^\omega)$, where $\omega\in [2,3]$ is the exponent for matrix multiplication.  Thus, every edge can be processed in $O(d^3)$ time.  All $n$ edges are processed sequentially for a total time of $O(nd^3)$.

In the backward pass, on each edge we must perform a shape commutation $E_L L = \tilde{L}E_L$, or $L\hat{E}_L = \hat{E}_L \tilde{L}$.  $\tilde{L}$ can be formed by reading off entries of $L$, so can be formed in $O(d^2)$ time.  We then form the product of two lower-triangular matrices in $O(d^\omega)$ time, and then pass the product $L$ matrix to the left.  This might require inversion of the $L$ matrix, which can be achieved in $O(n^\omega)$ time.  All $n$ edges are processed sequentially for a total time of $O(nd^\omega)$.

The forward pass dominates the complexity, for a time of $O(nd^3)$.
\end{proof}

\subsection{Parallelization through Divide-and-Conquer}\label{sec:parallel_alg}
We can also parallelize the algorithm for computing the barcode factorization of a quiver. The protocol of matrix factorizations and matrix passing will be different.
\subsubsection{LQU and UQL Factorizations}
The $LQU$ factorization is different from the triangular factorizations introduced before, as it does not perform any pivoting and therefore there is no auxiliary permutation matrix produced. Instead, we sacrifice structure in the pivot matrix and obtain a general pivot matrix $Q$ as opposed to echelon pivot matrices $E$.

\begin{algorithm}
\caption{LQU factorization}
\label{alg:LQU}
\begin{algorithmic}[1]
\STATE{{\bf Input:} $m\times n$ matrix $A$}
\STATE{{\bf Result:} Factorization $A = LQU$}
\STATE{$L \gets I_m$, $Q \gets A$, $U \gets I_n$}
\STATE{$ j = 1$}
\WHILE{$j <= n$} 
    \IF{column $j$ has a non-zero in a non-pivot row, let the first such row be $i$}
        \STATE{Zero out all elements in non-pivot rows in column $j$ below $i$, recording operations in $L$.}
        \STATE{Mark $i$ as pivot row}
        \STATE{$j = j+1$}
    \ELSE
        \STATE{$j = j+1$}
    \ENDIF
\ENDWHILE
\STATE{$ i = 1$}
\WHILE{$i <= m$} 
    \IF{row $i$ is a pivot-row with pivot at $j$}
        \STATE{Zero out all elements after $j$ in row $i$, recording operations in $U$.}
        \STATE{$i = i+1$}
    \ELSE
        \STATE{$i = i+1$}
    \ENDIF
\ENDWHILE
\end{algorithmic}
\end{algorithm}

\begin{proposition}
\cref{alg:LQU} computes a factorization $A = LQU$
\end{proposition}
\begin{proof} In order to see that the algorithm is correct, we first note that both the row operations and column operations are triangular, i.e. row $i$ is used to eliminate rows at positions greater than $i$ and column $j$ is used to eliminate columns at positions greater than $j$. Thus the recorded $L$ and $U$ matrices are indeed lower and upper triangular respectively.

Now it is left to prove that the resulting matrix $Q$ has pivot structure. If we prove that the only non-zeros at the end of the algorithm are the pivots then we are done, as pivots are chosen such that no two of them share a row or column. We will argue by contradiction: suppose there is a non-zero that was not eliminated at the end of the algorithm. It has to be either in a non-pivot row or its column position is before a pivot in a pivot row, otherwise it would have been eliminated by the column operations. Such an element cannot exist as it should have been eliminated by row operations by a pivot above it in the same column. The pivot cannot be below as it would imply that we did not pick the first non-zero in a non-pivot row when choosing the pivot for this column.
\end{proof}
\begin{proposition}\label{prop:LQU_time}
The $LQU$ factorization of a $m\times n$ matrix $A$ can be computed in $O(mn \max(m,n))$ time.
\end{proposition}
\begin{proof}
In the first while-loop of \cref{alg:LQU}, we loop over columns.  Each column eliminates $O(m)$ entries via row operations.  Each row operation takes $O(n)$ time (by operating on each column), so processing a column takes $O(mn)$ time.  Multiplying this bound by $n$ columns, we obtain $O(mn^2)$ time for the first while-loop.

The second while-loop, we loop over rows.  Each row may eliminate up to $n$ entries using column operations. Each column operation takes $O(m)$ time (operating on each row), so processing a row takes $O(mn)$ time.  Multiplying by $m$ rows, we obtain $O(m^2 n)$ time for the second while-loop.

Finally, we combine the bounds on both while-loops to obtain at total time of $O(mn \max(m,n))$.
\end{proof}

We can derive the UQL factorization from the LQU factorization using conjugation by J-matrices.
\begin{align*}
    A &= JJAJJ\\
    &= JLQUJ &LQU \text{ of } JAJ\\
    &= (JLJ)(JQJ)(JUJ) &\text{inserting } I = JJ \text{ between matrices}\\
    &= \tilde{U} \tilde{Q} \tilde{L}
\end{align*}

As a corollary to \cref{prop:LQU_time}, the $UQL$ factorization can be computed in $O(mn\max(m,n))$ time, since $J$-conjugation takes $O(mn)$ time.  Alternatively, one can make straightforward modifications to \cref{alg:LQU} to compute the UQL factorization.

\subsubsection{E Matrix transformations} \label{sec:emat_tran}
In this section we will see how we can factorize any pivot matrix $Q$ into a permutation and an echelon pivot matrix

\begin{proposition}\label{prop:pivot_trans}
Given any pivot matrix Q, we can rewrite it as the following
\begin{align}
    Q &= E_LP \\
    Q &= PE_U \\
    Q &= \hat{E}_UP \\
    Q &= P\hat{E}_L \\
\end{align}
where $P$ is an appropriate permutation matrix.
\end{proposition}
\begin{proof}
We apply the $LE_LUP$, $PLE_UU$, $U\hat{E}_ULP$ and $PU\hat{E}_LL$ factorizations to $Q$ and note that the triangular matrices are just the identity matrices.
This is because the triangular matrices are produced to eliminate one entry with another entry in the same row or column, but such a situation cannot occur in a pivot matrix $Q$, so only permutation operations are performed in the factorization, resulting in a permutation matrix and an echelon pivot matrix.
\end{proof}

If the pivot matrices $Q, E$, and $P$ are stored in a sparse format (such list of lists format), the factorizations in \cref{prop:pivot_trans} can be computed in $O(d \log d)$ time, where $d = \max(m,n)$ and $Q$ is a $m\times n$ pivot matrix, by simply sorting the nonzero entries in $Q$ by either row or column index.

\subsubsection{Divide and conquer}
Now we will show how we can divide a type A quiver into two parts and perform the computation in parallel for each of the parts. We divide a quiver representation into two parts at position $m$ (in practice, we choose $m = \lfloor n/2 \rfloor$ for a quiver of length $n$). For simplicity, we'll visualize the scheme on a persistence-type quiver, with the understanding that reversed arrows will make appropriate modifications to factorizations.
\begin{center}
    \begin{tikzcd}[row sep=small, column sep=huge]
    \cdot \ar[r, leftarrow, "A_0"] &\cdot \cdots\ar[r, leftarrow, "A_{m-1}"] & \cdot\ar[r, leftarrow, "A_{m}"]& \cdot\ar[r, leftarrow, "A_{m+1}"] &\cdots \cdot\ar[r, leftarrow, "A_{n-1}"] & \cdot
    \end{tikzcd}
\end{center}

We can now perform two versions of the sequential algorithm in parallel. For quiver $Q_\alpha$, we will use the rightward-initial algorithm (\cref{alg:typeA_alg_ri}), and for quiver $Q_\beta$, we will use the leftward-initial variant. Notice that the terminal matrices are collected at $A_m$.
\begin{center}
    \begin{tikzcd}[row sep=small, column sep=huge]
    \cdot \ar[r, leftarrow, "L_\alpha E_0"] &\cdot \cdots\ar[r, leftarrow, "E_{m-1}"] & \cdot\ar[r, leftarrow, "U_\alpha P_\alpha A_{m} P_\beta L_\beta"]& \cdot\ar[r, leftarrow, "E_{m+1}"] &\cdots \cdot\ar[r, leftarrow, "E_{n-1} U_\beta"] & \cdot
    \end{tikzcd}
\end{center}

We can now multiply out the matrices around $A_m$, and perform an $LQU$ factorization using \cref{alg:LQU}
\begin{center}
    \begin{tikzcd}[row sep=small, column sep=huge]
    \cdot \ar[r, leftarrow, "L_\alpha E_0"] &\cdot \cdots\ar[r, leftarrow, "E_{m-1}"] & \cdot\ar[r, leftarrow, "L_m Q_m U_m"]& \cdot\ar[r, leftarrow, "E_{m+1}"] &\cdots \cdot\ar[r, leftarrow, "E_{n-1} U_\beta"] & \cdot
    \end{tikzcd}
\end{center}

The matrices $E_0$ to $E_{m-1}$ were produced by the rightward-initial algorithm, so they can be used to commute the $L_m$ factor all the way to the left. This process is very similar to the second sweep of the rightward-initial algorithm. Similarly, the matrices $E_{m+1}$ to $E_{n-1}$ were produced by the leftward initial algorithm, and can be used to commute the $U_m$ factor towards the right in a manner similar to the second sweep of the leftward-initial algorithm.
\begin{center}
    \begin{tikzcd}[row sep=small, column sep=huge]
    \cdot \ar[r, leftarrow, "\tilde{L}_\alpha E_0"] &\cdot \cdots\ar[r, leftarrow, "E_{m-1}"] & \cdot\ar[r, leftarrow, "Q_m"]& \cdot\ar[r, leftarrow, "E_{m+1}"] &\cdots \cdot\ar[r, leftarrow, "E_{n-1} \tilde{U}_\beta"] & \cdot
    \end{tikzcd}
\end{center}

We can perform change of bases on $V_0$ and $V_n$ to remove the $\tilde{L}_\alpha$ and $\tilde{U}_\beta$ terms, and obtain a valid barcode factorization of $Q_\gamma$.  In order to apply this procedure recursively, we want to convert this into the result of either a leftward-initial or rightward-initial algorithm. This can be done by transforming the pivot matrix $Q_m$ to the appropriate form in \cref{prop:pivot_trans}, and then propagating the permutation matrix either leftwards or rightwards, repeatedly using the appropriate factorization in \cref{prop:pivot_trans}. 

Note that the results may not be exactly equal to the result obtained from applying either a rightward-initial or leftward-initial algorithm on the entire quiver. While the echelon matrices are of the right shape, the terminal matrices $P\tilde{U}_{\beta}$ or $\tilde{L}_{\alpha}P$ appear in reversed shape order.  This does not affect our ability to merge barcode factorizations at a higher level.

We can now apply this recursively until the size of the quiver is small enough that it is more efficient to apply one of the sequential algorithms.  \cref{alg:dq_barcode} gives the general version of the algorithm with arbitrary edge directions.

\begin{algorithm}[htb] 
\caption{Divide and Conquer Barcode Form}
\label{alg:dq_barcode}
\begin{algorithmic}[1]
\STATE{{\bf Input:} Type $A_{n+1}$ quiver representation.}
\STATE{{\bf Result:} Barcode form of quiver representation in either leftward or rightward form.}
\STATE{choose $m \in {0,\dots,n-1}$}
\STATE{Compute leftward barcode form of $A_0, \dots, A_{m-1}$}
\STATE{Compute rightward barcode form of $A_{m+1}, \dots, A_{n-1}$}
\STATE{Pass change of bases to $A_m$ to form $\tilde{A}_m$.}
\IF{arrow $m$ is $\leftarrow$}
    \STATE{Form $\tilde{A}_m = L_m Q_m, U_m$}
\ELSE
    \STATE{Form $\tilde{A}_m = U_m Q_m, L_m$}
\ENDIF
\STATE{Commute $L_m$ to the left, $U_m$ to the right}
\IF{Rightward Form}
    \STATE{Form $Q_m = E_L P$, propagate $P$ rightward}
\ENDIF
\IF{Leftward Form}
    \STATE{Form $Q_m = P E_U$, propagate $P$ leftward}
\ENDIF
\end{algorithmic}
\end{algorithm}

\begin{proposition}\label{prop:dq_complexity_sequential}
Let $d$ denote the maximum vector space dimension in a quiver representation of type $A_{n+1}$.  When \cref{alg:dq_barcode} is used recursively, it computes the barcode form in $O(d^3 n + (d\log d) (n \log n))$ time when run sequentially.
\end{proposition}
\begin{proof}
Let $T(n)$ denote the time it takes to compute either the leftward or rightward barcode form on $n$ edges.  Because we use \cref{alg:dq_barcode} recursively, we can write
$$T(n) = 2T(\frac{n}{2}) + f(n)$$
where $f(n)$ is the time it takes to combine the two halves of the quiver representation and transform into either a leftward or rightward form.

Note that Forming $\tilde{A}_m$ and the LQU factorization will take $O(d^3)$ time.  
We then commute $L_m$ all the way left in $O(d^2)$ time, and $U_m$ all the way right in $O(d^2)$ time using \cref{prop:comm_invert} and maintaining the products of $E_L$ and $\hat{E}_L^\dagger$ and the products of $E_U$ and $\hat{E}_U^\dagger$ respectively.  Finally, we transform the quiver representation into either rightward or leftward form.  This requires transforming $O(\frac{n}{2})$ pivot matrices into factorizations as in \cref{prop:pivot_trans}, each of which can be computed in $O(d \log d)$ time, and updating the products of the $E$-type matrices, which can be done at a cost of $O(d)$ time on each edge.

Combining the above, we have
$$T(n) = 2T(\frac{n}{2}) + O(\frac{n}{2} d\log d + d^3)$$
We can now apply the master theorem for recursion \cite{CLRS} to obtain the asymptotic run time.  The $O(d^3)$ computation is done at every leaf of the recursion for a time of $O(n d^3)$, and the exponent on $n$ in propagation of the permutation matrix is on par with the critical exponent of the recursion $1 = \log_2(2)$, so we obtain a factor of $O(d\log d n\log n)$.

\end{proof}

\cref{prop:dq_complexity_sequential} indicates that there isn't an advantage to running the divide and conquer algorithm sequentially, and in fact we expect it to run asymptotically slower as $n$ increase.  However, we do see an advantage if we parallelize the computation, as seen in the following proposition

\begin{proposition}\label{prop:dq_complexity_parallel}
Let $d$ denote the maximum vector space dimension in a quiver representation of type $A_{n+1}$.  When \cref{alg:dq_barcode} is used recursively, it computes the barcode form in $O((d \log d) n + d^3 \log n)$ time when recursion is run in parallel with an unlimited number of parallel processes available.
\end{proposition}
 \begin{proof}
 The analysis is much the same as in \cref{prop:dq_complexity_sequential}, but because we can use parallelism when computing the barcode form of the two sub-quivers, we obtain
 $$T(n) = T(\frac{n}{2}) + O((d\log d) n + d^3)$$
 We have to perform $O(d^3)$ work at $\log n$ levels of recursion, and now the $(d \log d) n$ term dominates the critical exponent of the recursion in the master theorem.  Thus, we obtain the bound.
 \end{proof}
 
 From \cref{prop:dq_complexity_parallel}, we see an advantage to using \cref{alg:dq_barcode} over the sequential \cref{alg:typeA_alg_ri} when parallelism is available, particularly when $d$ is large.  There are a variety of factors that will determine the true trade-off, such as how many parallel processes are available, and how sparsity in the matrices affects computation.  Computational experiments can be found in \cref{sec:performance}.

\subsection{Correctness and Uniqueness of the Barcode Factorization}\label{sec:alg_proof}

We have presented sequential and divide and conquer parallel algorithms to produce a barcode matrix $\Lambda$ from the companion matrix $A$ of a finite dimensional quiver representation of type $A_n$.  In this section we'll show the algorithm produces a quiver isomorphism $A \cong \Lambda$, and that $\Lambda$ uniquely defines the quiver isomorphism class.

\begin{proposition}
Every finite dimensional quiver representation of type $A_n$ has a barcode form.
\end{proposition}
\begin{proof}
This follows immediately from the existence of the LEUP and PLEU factorizations at each step, and the commutations established.
\end{proof}

\begin{proposition}
The barcode factorization $A = B \Lambda B^{-1}$ is a quiver isomorphism.
\end{proposition}
\begin{proof}
We have shown that the factorizations at each step exist, and by \cref{lem:diagram_matrix} each matrix passing step is a quiver isomorphism.
\end{proof}

We can state the fact that a barcode determines the isomorphism class of a zigzag or persistence module in terms of the barcode form:

\begin{theorem}\label{thm:barcode_form_unique}
The barcode form $\Lambda$ of a quiver representation of type $A_n$ uniquely determines its isomorphism class.
\end{theorem}
\begin{proof}
This follows from the fact that the barcode form encodes the interval indecomposables in \cref{prop:barcode_form}, and the fact that the interval indecomposables uniquely determine the isomorphism class \cite{gabrielI, ZZtheory2010}.

\end{proof}

This means that the matrix $\Lambda$ will not depend on the choice of basis used to form the companion matrix.
Note that the uniqueness of the Jordan normal form for companion matrices is a special case of \cref{thm:barcode_form_unique} that applies to persistence-type quivers, following \cref{sec:type_a_quiver}.

\section{Experiments and Applications}\label{sec:applications}

We'll now showcase several experiments which demonstrate the flexibility of our methodology.  We also investigate the performance of the implementation of our algorithms in the BATS software, demonstrating parallel speedups.

\subsection{Subsets and Rips Complexes}\label{sec:rips_subset_zigzag}

Let $\bX_0,\bX_1,\dots\subseteq X$ be finite subsets of a space, and $r_0,r_1,\dots \in \RR$ be parameters for the Rips construction.  We can construct a zigzag through unions of subsets
\begin{equation}
\begin{tikzcd}[column sep=small]
\calR(\bX_0; r_0) \ar[r, hookrightarrow] &\calR(\bX_0\cup \bX_1, \max\{r_0,r_1\}) \ar[r,hookleftarrow] & \calR(\bX_1; r_1) \ar[r,hookrightarrow] &\dots
\end{tikzcd}
\end{equation}
The case where $r = r_0 = r_1 = \dots$ was proposed in \cite{ZZtheory2010} as a ``topological bootstrapping'' method to investigate the robustness of homological features for samples $\bX_i$ sampled uniformly from some larger point cloud $\bX$, and initial investigations were performed by Tausz \cite{tausz2012}.  A single value of $r$ is also used in the case of \cite{ZZalg2009} where the $\bX_i$ come from thickened level sets of a map $f:\bX\to \RR$
$$
\bX_i = f^{-1}((a_i, b_i))
$$
where $(a_i,b_i) \cap (a_{i+1}, b_{i+1}) \ne \emptyset$.

Another application of a zigzag of this form is to approximate the persistent homology of the full Rips filtration.  This was originally implemented by Morozov in his Dionysus software \cite{Dionysus2}, and several variants with theoretical guarantees are investigated in \cite{ZigzagZoologyRips2015}.  In this situation, $\bX_0 \subseteq \bX_1 \subseteq \cdots$, and $r_0 \ge r_1 \ge \cdots$.  An example can be found in \cref{fig:dmzz}, where we use this construction on 200 points sampled from the unit circle, and compare to the persistent homology of a Rips filtration.  Note that while the zigzag and persistent homology results are not identical, they display the same qualitiative information.  The zigzag computation in BATS is faster than the persistence computation performed by the highly optimized Ripser software \cite{Ripser19} (0.08 seconds vs 0.12 seconds), indicating that our methodology may be employed to effectively use the constructions in \cite{ZigzagZoologyRips2015} in situations which require extreme speed.  In \cref{fig:bats_ripser}, we see a comparison of the time it takes to compute the discrete Morozov zigzag versus

\begin{figure}
    \centering
    \includegraphics[width=0.49\linewidth]{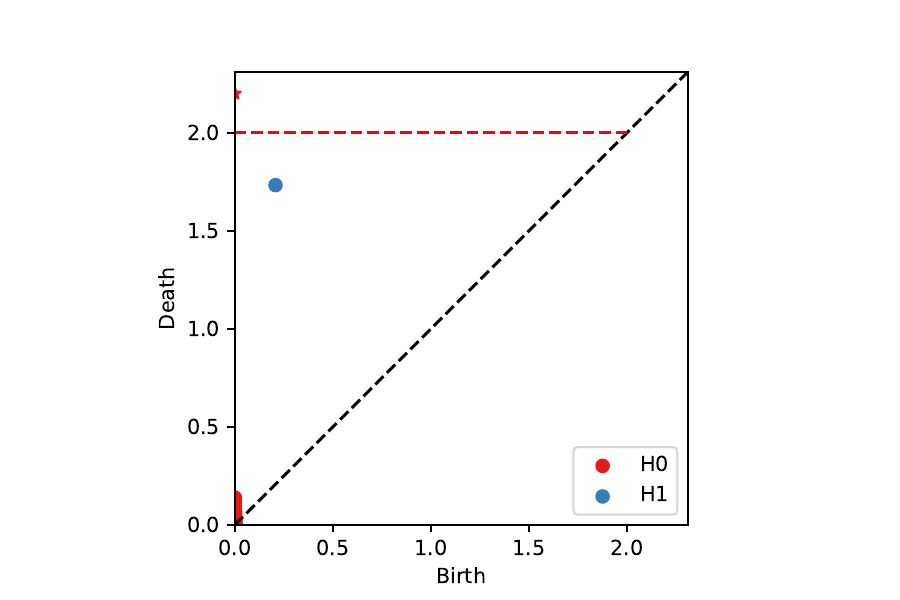}
    \includegraphics[width=0.49\linewidth]{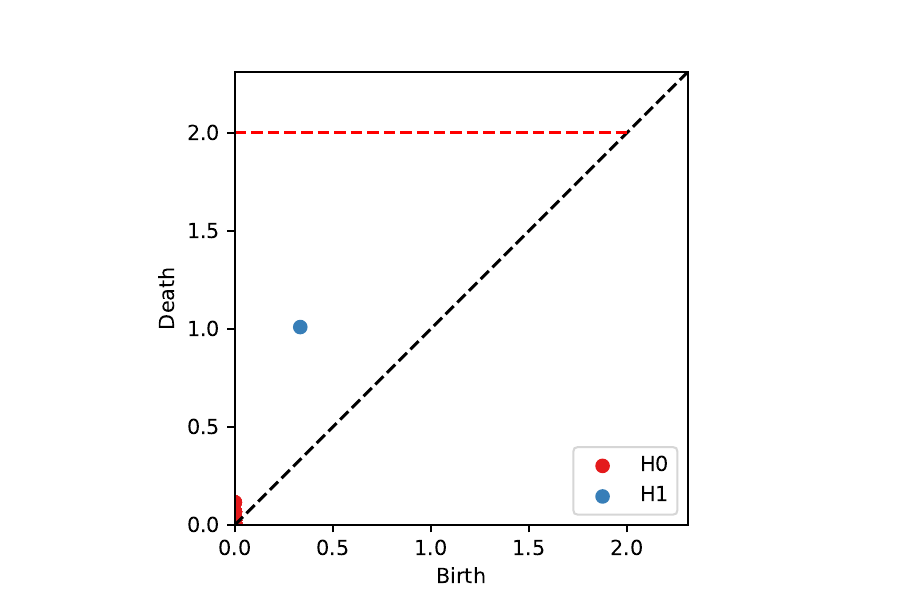}
    \caption{Left: persistence diagram for Rips filtration on 200 points sampled from the unit circle.  The unoptimized reduction algorithm runs in 20 seconds in BATS, and a version using compression \cite{bauerClearCompressComputing2014} takes 2.5 seconds.  The highly optimized Ripser software \cite{Ripser19} takes 0.12 seconds.  Right: An approximate persistence diagram created using the discrete Morozov zigzag construction in \cite{ZigzagZoologyRips2015} using the suggested parameters.  The zigzag computation takes approximately 0.08 seconds in BATS. While the birth and death times are not identical, both diagrams qualitatively display the same information, namely a single connected component and a robust $H_1$ class, agreeing with the homology of the circle.}
    \label{fig:dmzz}
\end{figure}

\begin{figure}
    \centering
    \includegraphics[height=60mm]{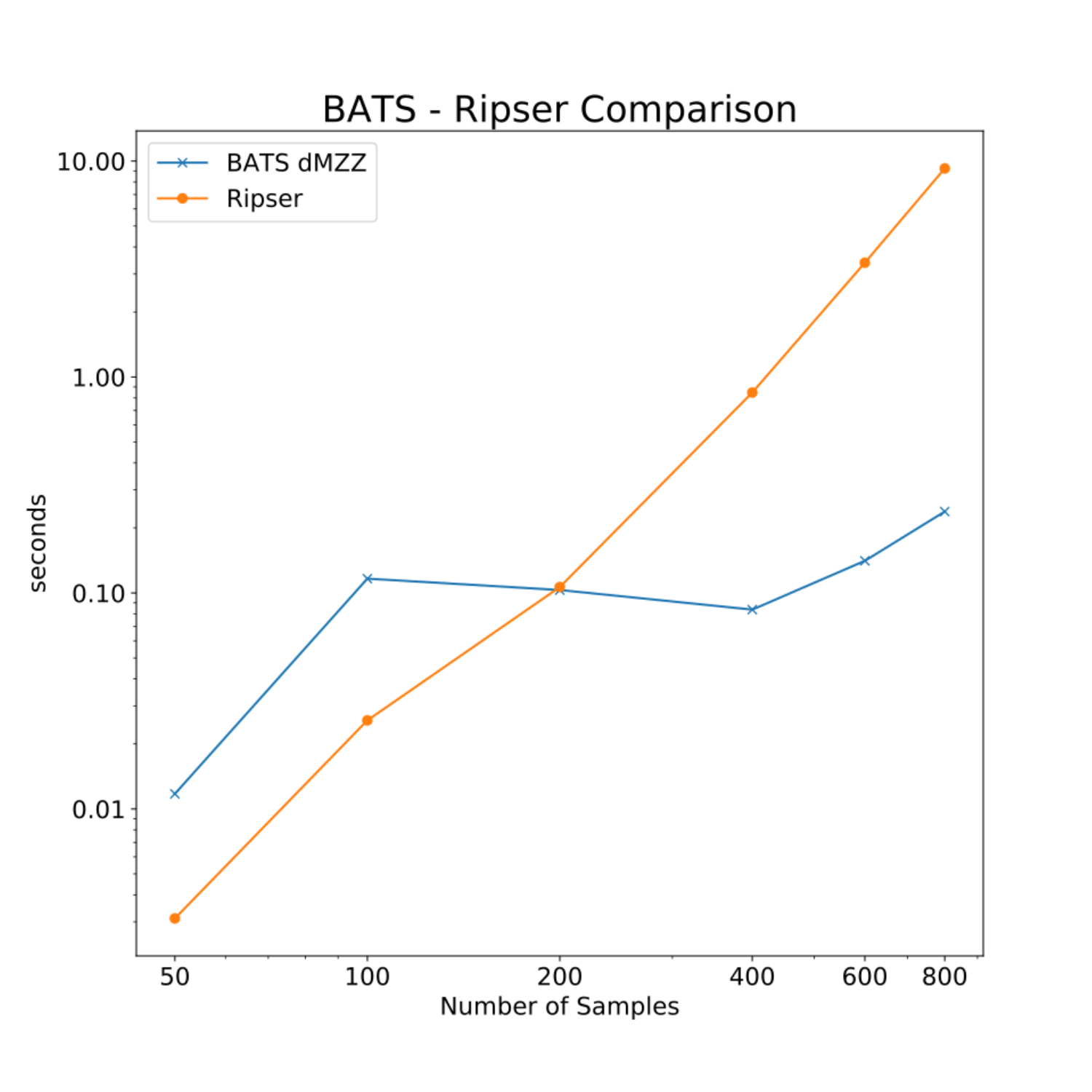}
    \caption{Time comparison between computing the discrete Morozov zigzag (dMZZ) construciton of \cite{ZigzagZoologyRips2015} using BATS versus computing persistent homology using Ripser \cite{Ripser19}.  Samples are drawn uniformly from the unit circle. For computations using over 200 samples the zigzag computation is faster.  Every point indicates the median runtime over 10 runs.}
    \label{fig:bats_ripser}
\end{figure}

\subsection{Bivariate Nerve}\label{sec:bivariate_nerve}

Another application of zigzag homology is to investigate the relationship between algebraic features in nerves of two or more covers.  This question arises naturally when generating covers algorithmically, where one might wonder how sensitive the Nerve is to choices that might be made or randomness in initialization.  A related question is the stability of witness complexes to the choice of landmark set, and a {\em bivariate witness} complex was proposed in \cite{ZZtheory2010} and subsequently investigated by Tausz \cite{tausz2012}. 
\begin{definition}\label{def:bivariate_nerve}
Given two covers $\calU$, $\calV$ of a space $X$, the bivariate cover is the fibered product $\calU \times_X \calV \subseteq \calU \times \calV$.  Explicitly,
$$\calU\times_X \calV = \{ U \times V \mid U\in \calU, V\in \calV, U\cap V \ne \emptyset\}$$
and $\calU \times \calV$ can be identified as the interesection $U \cap V$ to form another cover of $X$.  We will denote the nerve of $\calU \times_X  \calV$ as $\calN(\calU, \calV)$.
\end{definition}

Due to the product structure on $\calU \times \calV$, there are projection maps $p_U: U \times V \mapsto U$ and $p_V: U\times V \mapsto V$.  These maps extend to simplicial maps on the nerves:
\begin{equation}
\begin{tikzcd}
\calN(\calU) \ar[r, leftarrow, "p_U"] & \calN(\calU, \calV) \ar[r, "p_V"] &\calN(\calV)
\end{tikzcd}
\end{equation}

We will consider covers based on landmark sets.  Briefly, we choose a landmark set $\bL \subset \bX$, and assign each point in $\bX$ to the $k$-nearest landmarks.  An example is performed in \cref{fig:cover_zz}, where covers of a noisy circle are created.  We see that if each point belongs to only 2 sets that while there is a single long $H_1$ class, there are also many short-lived $H_1$ classes since there are no non-empty 3-way intersections to fill in small holes.  However, if every point is assigned to 3 sets, then there is a single $H_1$ class that persists the length of the diagram.

\begin{figure}
    \centering
    \includegraphics[width=0.49\linewidth]{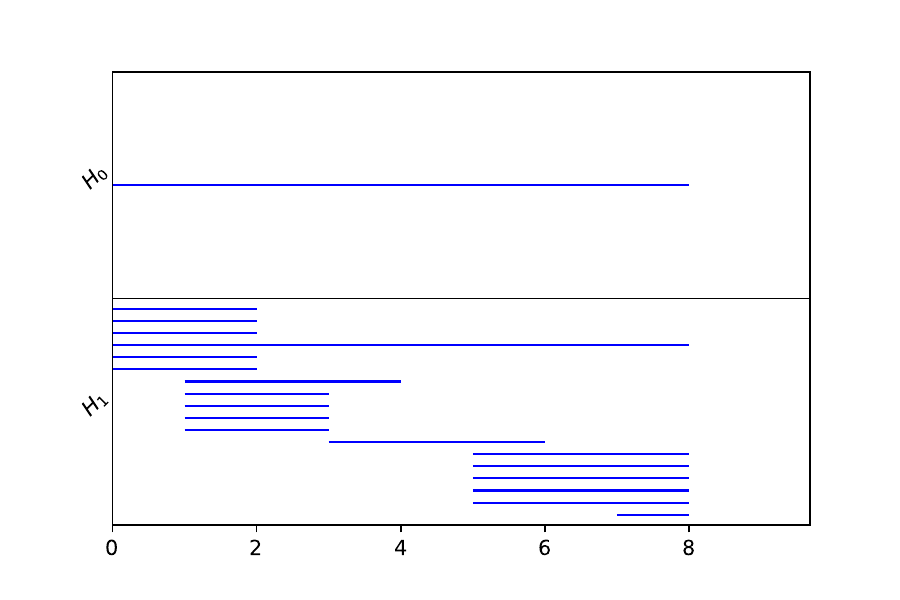}
    \includegraphics[width=0.49\linewidth]{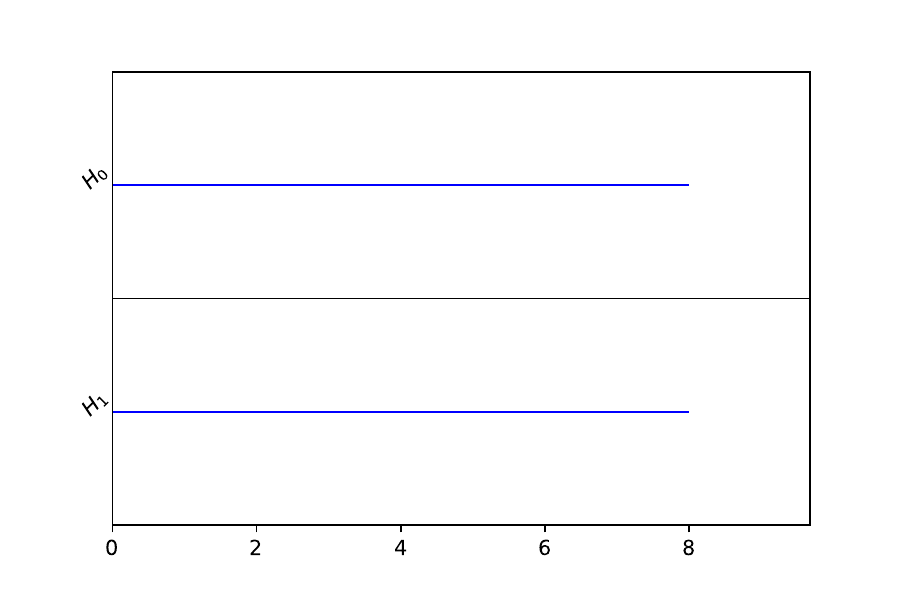}
    \caption{Zigzag barcodes of bivariate Nerve diagram on 5 covers of 500 points on the unit circle.  Covers are computed by selecting 20 random landmarks.  Left: each point is assigned to closest 2 landmarks.  Right: each point is assigned to closest 3 landmarks.  Note that both diagrams have single long bars in dimensions 0 and 1, agreeing with the homology of the circle.}
    \label{fig:cover_zz}
\end{figure}

\subsection{Sierpinski Triangle}\label{sec:sierpinksi_triangle}

Our final application is an example of how persistent homology may be used with more general cell maps.  We'll consider a sequence of spaces converging to a Sierpinksi triangle:

\begin{equation}\label{eq:sierpinski_maps}
\begin{tikzpicture}
    \node[circle,fill=blue, inner sep=0pt, minimum size=4pt] (0) at (0,0) {};
    \node[circle,fill=blue, inner sep=0pt, minimum size=4pt] (1) at (0,2) {};
    \node[circle,fill=blue, inner sep=0pt, minimum size=4pt] (2) at (2,0) {};
    \draw [thick, draw=blue] (0) -- (1) {};
    \draw [thick, draw=blue] (0) -- (2) {};
    \draw [thick, draw=blue] (1) -- (2) {};
    
    \draw [thick, draw=gray, ->] (1.5, 1) -- (2.5,1) {};
    
    \node[circle,fill=blue, inner sep=0pt, minimum size=4pt] (00) at (3,0) {};
    \node[circle,fill=blue, inner sep=0pt, minimum size=4pt] (20) at (3,1) {};
    \node[circle,fill=blue, inner sep=0pt, minimum size=4pt] (01) at (3,2) {};
    \node[circle,fill=blue, inner sep=0pt, minimum size=4pt] (10) at (4,0) {};
    \node[circle,fill=blue, inner sep=0pt, minimum size=4pt] (02) at (5,0) {};
    \node[circle,fill=blue, inner sep=0pt, minimum size=4pt] (100) at (4,1) {};
    \draw [thick, draw=blue] (00) -- (01) {};
    \draw [thick, draw=blue] (00) -- (02) {};
    \draw [thick, draw=blue] (01) -- (02) {};
    \draw [thick, draw=blue] (20) -- (10) {};
    \draw [thick, draw=blue] (20) -- (100) {};
    \draw [thick, draw=blue] (10) -- (100) {};
    
    \draw [thick, draw=gray, ->] (4.5, 1) -- (5.5,1) {};
    
    \node[circle,fill=blue, inner sep=0pt, minimum size=4pt] (000) at (6,0) {};
    \node[circle,fill=blue, inner sep=0pt, minimum size=4pt] (020) at (6,1) {};
    \node[circle,fill=blue, inner sep=0pt, minimum size=4pt] (a) at (6,0.5) {};
    \node[circle,fill=blue, inner sep=0pt, minimum size=4pt] (b) at (6.5,0.5) {};
    \node[circle,fill=blue, inner sep=0pt, minimum size=4pt] (c) at (6.5,0) {};
    
    \node[circle,fill=blue, inner sep=0pt, minimum size=4pt] (001) at (6,2) {};
    \node[circle,fill=blue, inner sep=0pt, minimum size=4pt] (010) at (7,0) {};
    \node[circle,fill=blue, inner sep=0pt, minimum size=4pt] (d) at (7,0.5) {};
    \node[circle,fill=blue, inner sep=0pt, minimum size=4pt] (e) at (7.5,0.5) {};
    \node[circle,fill=blue, inner sep=0pt, minimum size=4pt] (f) at (7.5,0) {};
    
    \node[circle,fill=blue, inner sep=0pt, minimum size=4pt] (002) at (8,0) {};
    \node[circle,fill=blue, inner sep=0pt, minimum size=4pt] (0100) at (7,1) {};
    \node[circle,fill=blue, inner sep=0pt, minimum size=4pt] (g) at (6,1.5) {};
    \node[circle,fill=blue, inner sep=0pt, minimum size=4pt] (h) at (6.5,1.5) {};
    \node[circle,fill=blue, inner sep=0pt, minimum size=4pt] (i) at (6.5,1.0) {};
    
    \draw [thick, draw=blue] (000) -- (001) {};
    \draw [thick, draw=blue] (000) -- (002) {};
    \draw [thick, draw=blue] (001) -- (002) {};
    \draw [thick, draw=blue] (020) -- (010) {};
    \draw [thick, draw=blue] (020) -- (0100) {};
    \draw [thick, draw=blue] (010) -- (0100) {};
    
    \draw [thick, draw=blue] (a) -- (b) {};
    \draw [thick, draw=blue] (a) -- (c) {};
    \draw [thick, draw=blue] (b) -- (c) {};
    
    \draw [thick, draw=blue] (d) -- (e) {};
    \draw [thick, draw=blue] (d) -- (f) {};
    \draw [thick, draw=blue] (e) -- (f) {};
    
    \draw [thick, draw=blue] (g) -- (h) {};
    \draw [thick, draw=blue] (g) -- (i) {};
    \draw [thick, draw=blue] (h) -- (i) {};
    
    \draw [thick, draw=gray, ->] (7.5, 1) -- (8.5,1) {};
    
    \node () at (9,1) {\dots};

\end{tikzpicture}
\end{equation}

Where each map above sends vertices to the vertex in the same location in the image to the right.  Then each edge is subdivided in the image, introducing a new node and two new edges, and additional edges are introduced to fill in new interior triangles.

The maps are not simplicial because each edge is subdivided at each iteration.  One could also consider a filtration of simplicial complexes based on the finest Sierpinski mesh, but the advantage of the mapping construction is that it is easy to add another iteration of the subdivision to the end of the sequence without recomputing every space.  We compute the persistence barcode in \cref{fig:sierpinski_pd}.  There is a single connected component, and the $k$th iteration (starting at $k=0$) adds another $3^k$ $H_1$ classes.

\begin{figure}
    \centering
    \includegraphics[height=60mm]{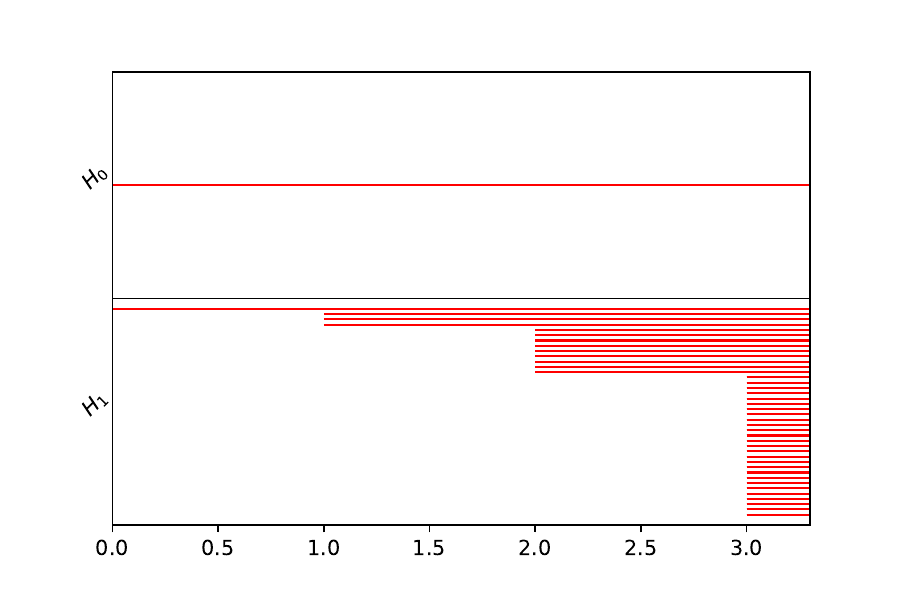}
    \caption{Persistence barcode showing 4 iterations of the sequence in \cref{eq:sierpinski_maps}. }
    \label{fig:sierpinski_pd}
\end{figure}

This is a fairly simple example in which the maps are easy to specify.  In general, it can be difficult to specify arbitrary cell maps, but there are potential algorithmic approaches such as those presented in \cite{nelsonThesis2020} that can be used to simplify the task.  An implementation is beyond the scope of this work, but our methodology opens up the possibility of using these constructions and general cellular maps in future work.

\subsection{Performance}\label{sec:performance}

We have now seen several examples in which zigzag and persistent homology might be computed using our methodology.  In this section, we'll discuss the performance of our algorithm and implementation in more detail.  The core BATS software is written in C++ with OpenMP used for parallelization. We have also written Python bindings for BATS, and used these to compare with Dionysus via its Python bindings. All timing results were obtained on an Intex Xeon CPU capable of running up to 24 threads simultaneously.

In \cref{fig:bats_dion}, we compare the runtime of our BATS software with Dionysus \cite{Dionysus2} using a zigzag diagram of Rips complexes, as introduced in \cref{sec:rips_subset_zigzag}.  Dionysus implements a version the algorithm described in \cite{ZZalg2009}, and the BATS software uses parallelization of the homology functor and the sequential quiver algorithm in these experiments.  We would like to draw attention to the drastic speedup achieved by BATS, which only improves as the length of the zigzag grows.  Note that this speedup can not be explained simply by parallelization, as there is already a large performance gap with only two samples (the zigzag has three spaces and two maps) where BATS can take advantage of at most 3 threads at a time.  The run time for BATS does not significantly increase until we take 16 samples (31 nodes in the zigzag), which is the first time that we have more spaces and maps than we have threads available for parallelization of the homology functor.  However, even after this point, especially for our experiments with 200 points, BATS still still sees an increasing performance gap compared to Dionysus.  There are many factors at play that determine the time it takes to execute these experiments, differing complexity bounds, and implementation-specific differences such as how BATS and Dionysus handle sparsity in linear algebra.

While BATS can be used very effectively in situations such as the sub-sample zigzag where large numbers of cells are added and removed in each space, we still expect Dionysus to be the better tool in situations in which the zigzag only adds and removes a small number of simplices from a large space at each node in the zigzag.  This is because our methodology would do enormous amounts of redundant computation to compute homology of each space in this situation, which would only be partially alleviated by parallelization if the zigzag is sufficiently long.  One might attempt to share the redundant computation between nodes, but such efforts are outside the scope of this work.

\begin{figure}
    \centering
    \includegraphics[width=0.49\linewidth]{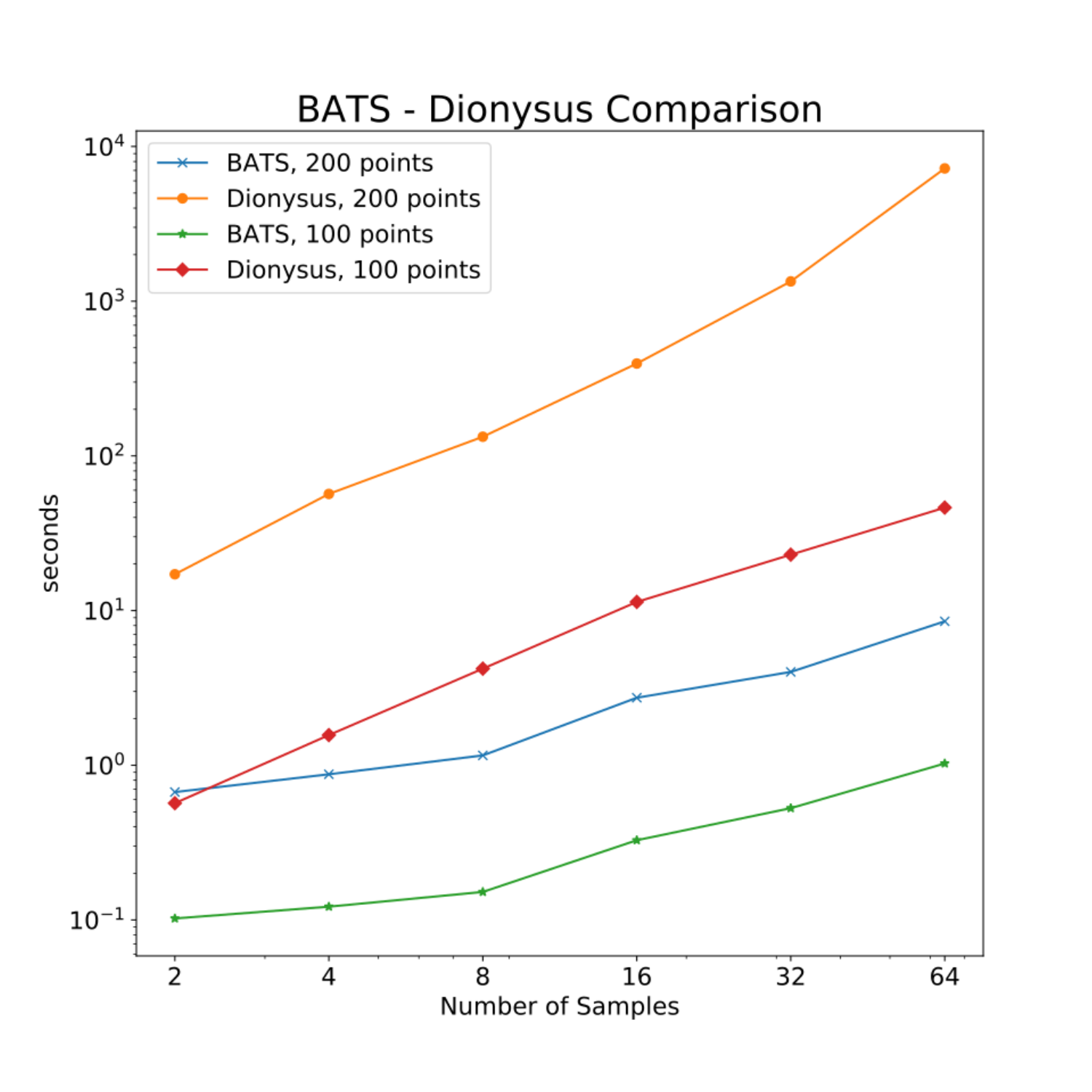}
    \includegraphics[width=0.49\linewidth]{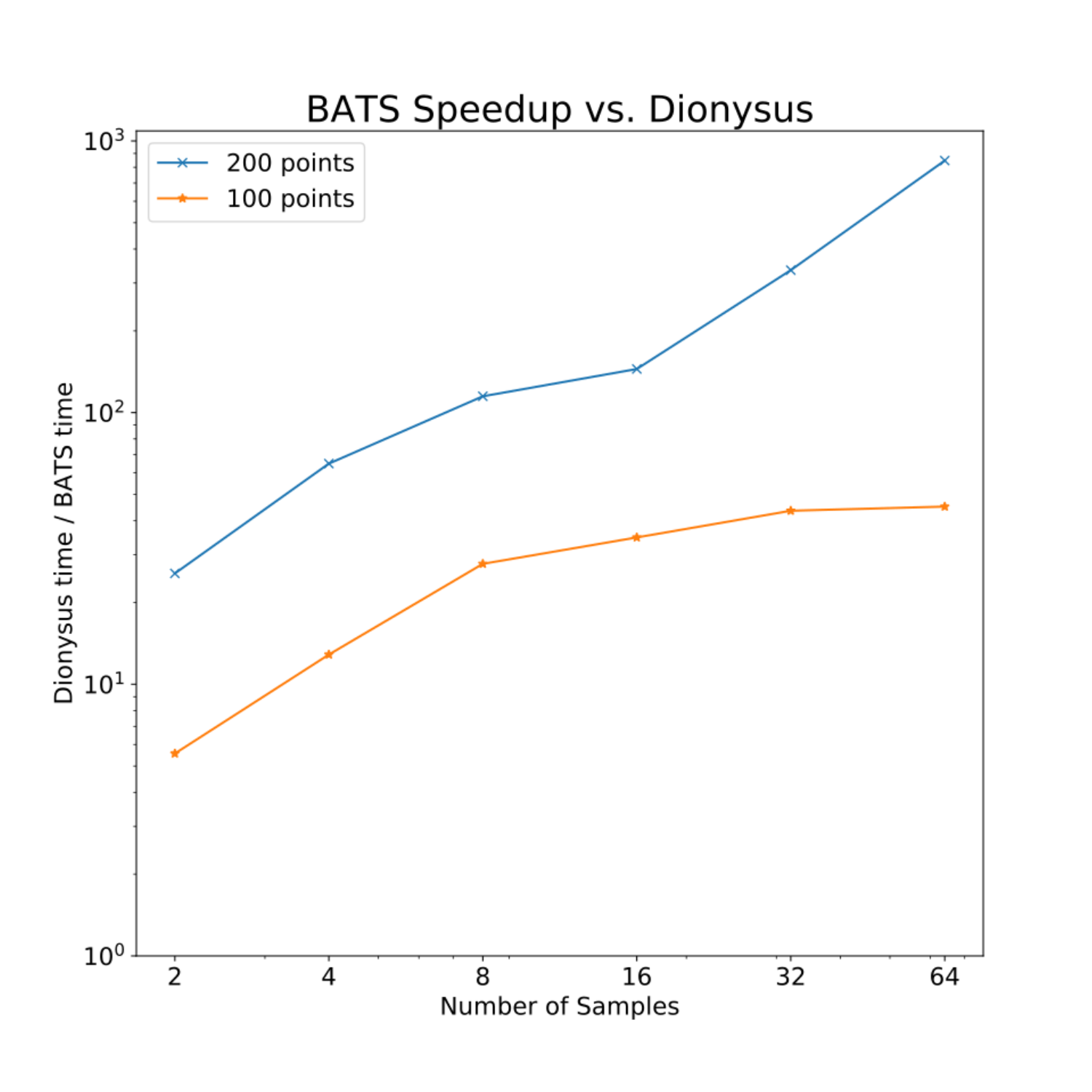}
    \caption{Time to compute the zigzag barcode of a zigzag diagram computed on sub-samples of a noisy circle and pairwise unions.  Normal noise with variance 0.1 is added to points sampled from a unit circle. In one experiment, samples of 100 points are obtained, and in another samples of 200 points are obtained.  The Rips parameter is set to $r=0.35$ in each case.   The horizontal axis indicates the number of samples used in the diagram.
    Left: time to compute zigzag homology for both BATS and Dionysus in each experiment.  Right: Speedup seen using BATS instead of Dionysus. BATS always outperforms Dionysus on these experiments, with a 1000x speedup on the most expensive computation.  See the text for further discussion.}
    \label{fig:bats_dion}
\end{figure}

\begin{figure}
    \centering
    \includegraphics[height=60mm]{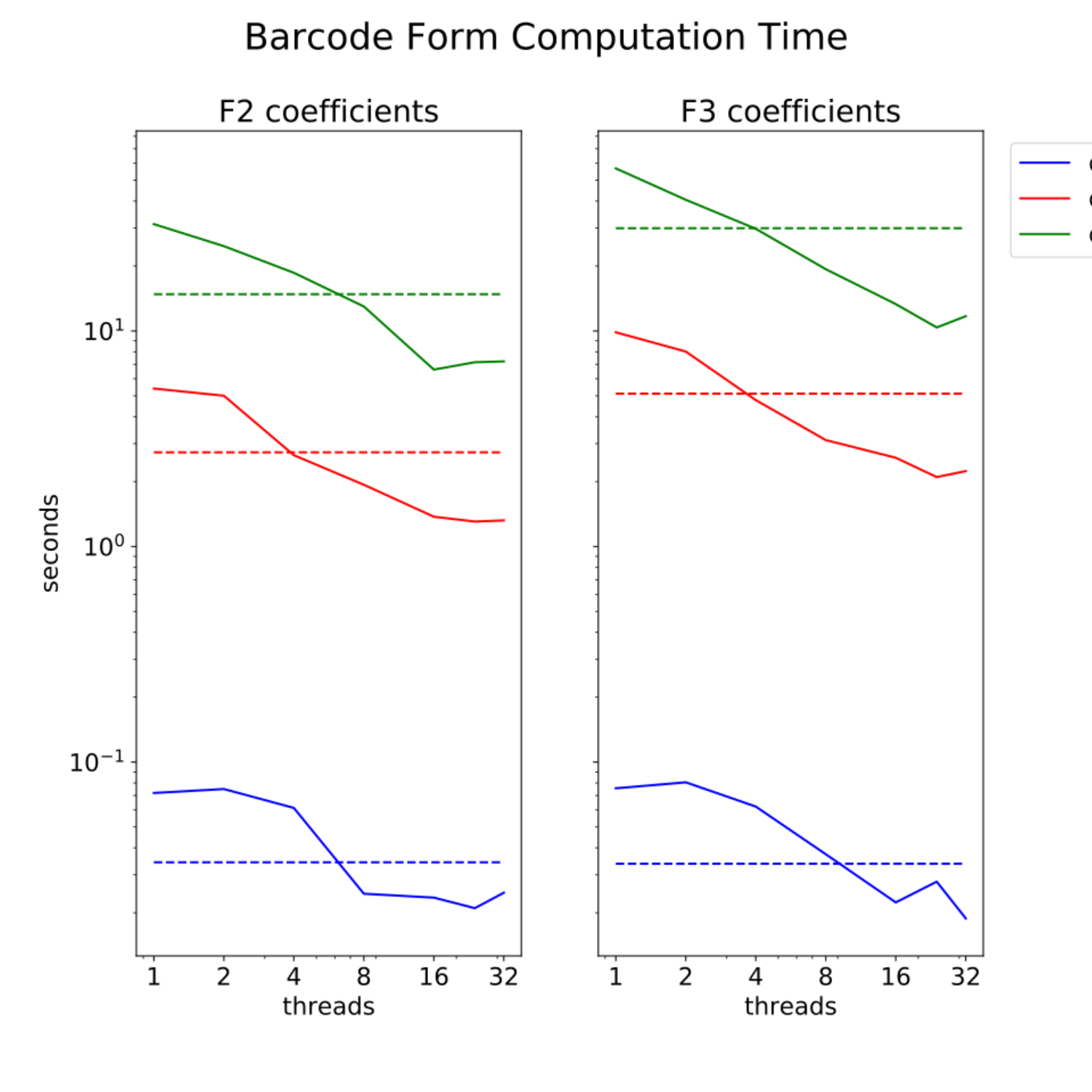}
    \caption{Time to compute barcodes of random quiver representations with fixed dimension $d$ in each vector space, as the number of OpenMP threads increases.  The length $n$ of the quiver is always $256$.  Left: the field $\FF_2$ is used, right: the field $\FF_3$ is used.  $d\times d$ matrices are generated on each edge where every entry has a 50\% chance of being non-zero.  The horizontal dashed lines are the time to run the sequential algorithm.  The solid lines indicate the time to run the divide-and-conquer algorithm, which decreases as the number of threads increase. We see that the divide and conquer algorithm is slower if only a few threads are used, but faster for larger numbers of threads.  See the text for further discussion. }
    \label{fig:dq_scaling1}
\end{figure}

We now turn to measuring the performance of the divide-and-conquer algorithm for computing a barcode, which is parallelized in BATS using OpenMP tasks.  In \cref{fig:dq_scaling1}, we see how the time to compute a barcode decreases as the number of available threads increases.  In the experiments performed the divide-and-conquer algorithm starts out as slower than the sequential implementation due to increased overhead, but starts to outperform the sequential algorithm at around 8 threads.  From this, we see that the sequential algorithm should be preferred when a sufficient number of processes are not available.  The threshold where the divide-and-conquer algorithm will be faster than sequential algorithm will be machine and problem dependent.

\begin{figure}
    \centering
    \includegraphics[height=60mm]{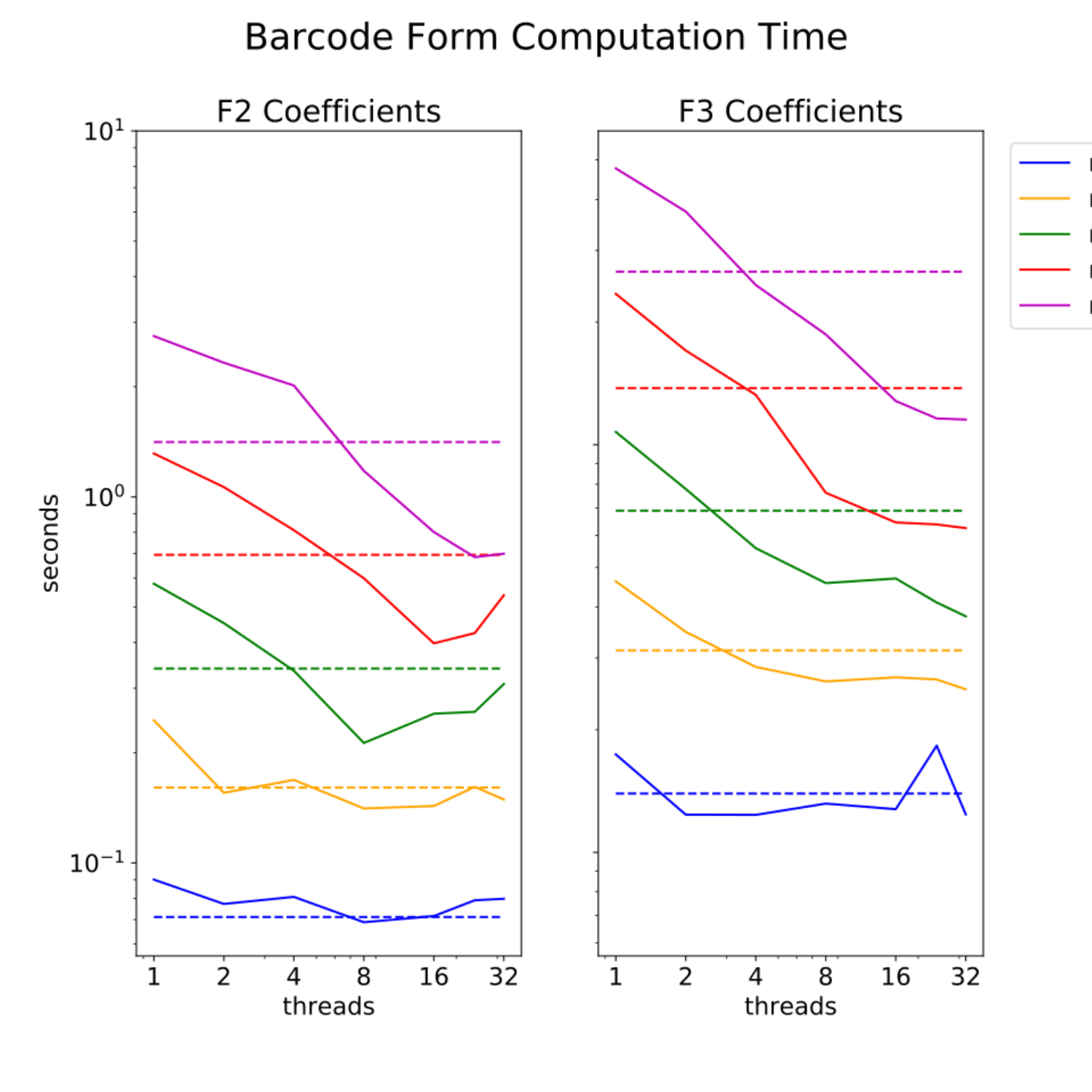}
    \caption{Time to compute barcodes of random quiver representations with fixed dimension $d=100$ in each vector space, as the number of OpenMP threads increases.  Left: the field $\FF_2$ is used, right: the field $\FF_3$ is used.  $d\times d$ matrices are generated on each edge where every entry has a 50\% chance of being non-zero.  The horizontal dashed lines are the time to run the sequential algorithm.  The solid lines indicate the time to run the divide-and-conquer algorithm, which decreases as the number of threads increase. If the length of the quiver, $n$ is small, the divide and conquer algorithm does not reach enough levels of recursion to see a noticeable benefit.  See the text for further discussion. }
    \label{fig:dq_scaling2}
\end{figure}

In \cref{fig:dq_scaling2}, we observe that because the divide-and-conquer algorithm must recurse to a sufficient depth to take full advantage of available parallelism, small problems (small values of $n$ in the plot) do not see a significant advantage compared to the sequential algorithm.  For larger values of $n$, the machine is able to use all threads, which results in a noticeable speedup.

\begin{figure}
    \centering
    \includegraphics[height=60mm]{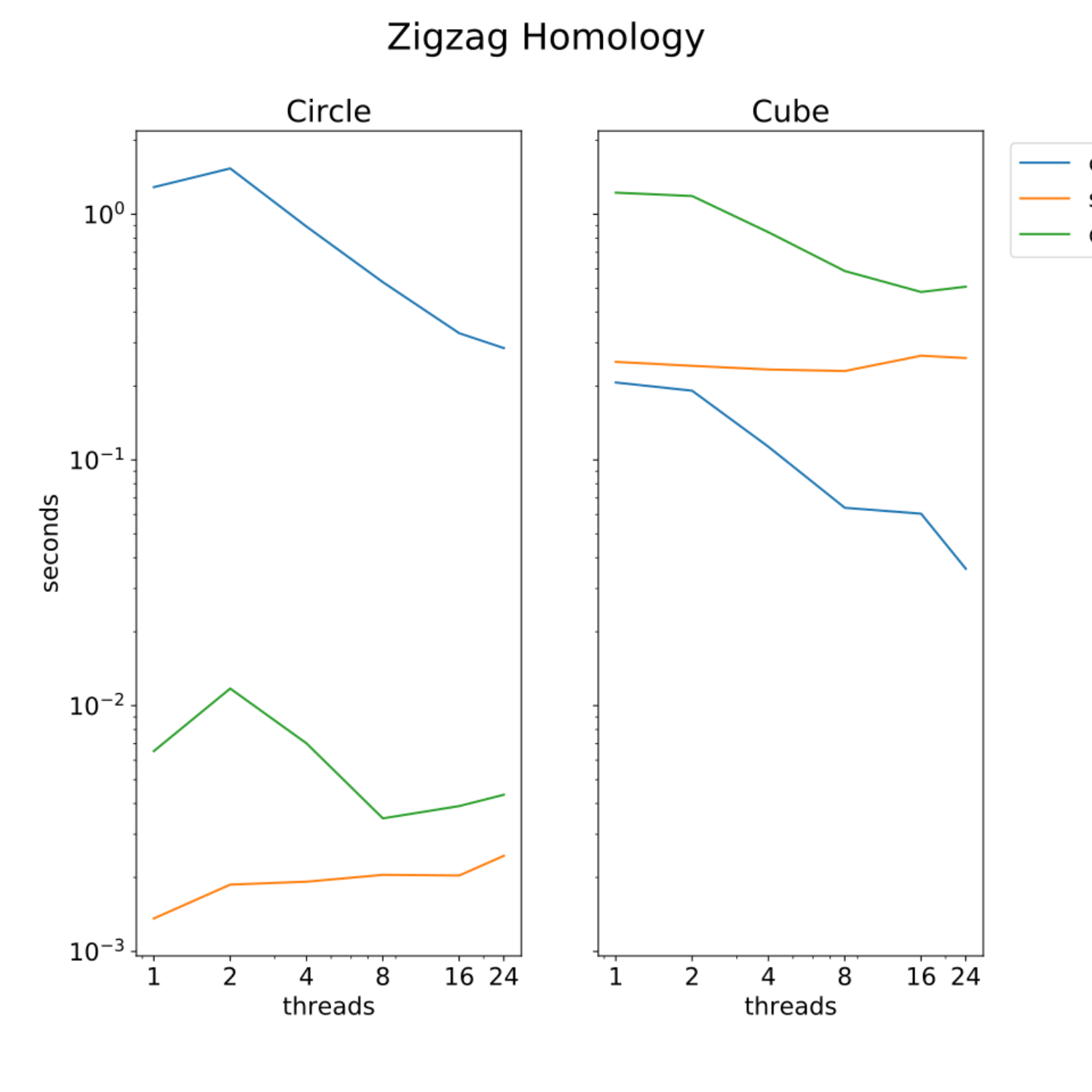}
    \caption{Time to compute zigzag homology of a zigzag diagram computed on 128 sub-samples and pairwise unions.  Left: samples obtained from a noisy circle, with 100 points per subset and Rips parameter $r=0.4$. Computing the homology functor (blue line) dominates the computation time, and because the homology is low-dimensional computing the barcode is very fast using either the sequential or divide and conquer algorithms.
    Right: samples obtained from the uniform distribution in the unit cube, with 400 points per subset and Rips parameter $r=0.03$.  Computing homology is very fast, but now the 0-dimensional homology is large so computing the barcode is more expensive.  See the text for further discussion. }
    \label{fig:perf_breakdown}
\end{figure}

Finally, we investigate how computation time differs on two different problems.  In \cref{fig:perf_breakdown}, we compute a zigzag barcode using Rips complexes built from subsamples of larger data sets.  Our methodology breaks down computation of a barcode into two steps: first, we compute induced maps on homology using the Chain and Hom functors, and second, we compute the barcode using either the sequential algorithm or divide-and conquer algorithm.

In the circle example (\cref{fig:perf_breakdown}, left hand side), the Rips parameter is chosen to be sufficiently high so that there will almost always be a single homology generator in both $H_0$ and $H_1$ at each node of the quiver, with identity maps along edges.  Computing homology is fairly expensive in this situation, and we see that the homology computation dominates the run time.  The time to compute homology decreases as we increase the number of available threads due to the embarassingly parallel algorithm to compute induced maps.  Both the sequential and divide-and-conquer algorithms are very fast by comparison, with the sequential option winning, likely due to the increased overhead necessary to employ divide-and-conquer.

In the cube example (\cref{fig:perf_breakdown}, right hand size), the Rips parameter is chosen to be small so there will be many connected components and $H_0$ will be roughly 200 dimensional for each subsample.  Computing homology is very cheap in this situation, and we see the homology computation is the fastest part of the computation. We still see a nice parallel speedup.  However, because homology is high-dimensional, the quiver algorithms take longer.  We see that the run time of the sequential algorithm does not depend on the number of threads, and that divide-and-conquer algorithm runs faster as we increase the number of threads.  In contrast to \cref{fig:dq_scaling1} and \cref{fig:dq_scaling2}, divide and conquer never runs faster than the sequential algorithm.  This is likely due to sparsity in the induced maps, which makes the true cost of factorizations cheap and decreases the relative advantage of using divide and conquer while keeping the overhead costs.

To conclude this section, we note that there are a variety of factors that go into choosing whether to use the sequential or divide-and-conquer algorithm for computing barcodes.  Our experiments demonstrate a parallel speedup for divide-and-conquer which pays off given a large enough problem size and sufficient parallelism.  The point at which divide-and-conquer will outperform the sequential algorithm will be machine dependent as well as problem dependent, and as we saw in \cref{fig:perf_breakdown}, sparsity in induced maps may play a role in determining which choice is better in practice.

One current limitation of our experiments is that they were performed on a single machine.  Our methodology could be employed in a massively parallel manner on a network of machines using MPI as well as OpenMP.  In situations where one might not even be able to fit the entire computation on a single machine,the divide and conquer approach may become more appealing in a broader class of situations.  These large-scale experiments are outside the scope of this work.

\section{Conclusion \& Future Directions}
We have presented a matrix factorization method for computing the indecomposables (barcodes) of finite type A quiver representations.  A natural question that arises is whether similar algorithms might exist for other types of quiver representations, which then might be used to extract information about different diagrams of spaces.  Representations on the other Dynkin diagrams almost certainly are amenable to similar techniques, as a result of Gabriel's theorem \cite{gabrielI}, but it would be necessary to go beyond sequential application of LU factorizations.  However, it is questionable whether type-D and type-E quiver representaitons would find applications in applied topology.  A more promising direction is to develop general algorithms for type-$\tilde{A}$ representations, whose underlying graphs discretize a circle.  Burghelea \cite{burgheleaNewTopologicalInvariants2018} has shown topologicially interesting applications, and in collaboration with Dey \cite{burgheleaTopologicalPersistenceCircleValued2013a} developed an algorithm that works in certain limited situations.

It is generally much more difficult to compute invariants of wild-type quiver representations, as has been seen in multi-parameter persistence \cite{carlssonTheoryMultidimensionalPersistence2009}.  However, one might hope that for certain {\em specialized} situations, the problem becomes more tractable.  This could be through choosing a specific field $\FF$ to work with, or by considering situations in which induced maps on homology satisfy certain properties.

We have considered using arbitrary induced maps on homology in our quiver representations.  In the special case of inclusion maps, enormous compression can be achieved by working at the chain level.  One might ask whether there are other situations in which such compression is achievable. Finally, we have worked exclusively with homology, whereas cohomology has additional structure in the form of various cohomology operations \cite{MosherTangora}. An interesting question is whether these cohomology operations could potentially aid in accelerating computations when computing barcodes for all dimensions simultaneously, both at the chain level, and at the level of induced maps.

\begin{acknowledgements}
 We thank Jonathan Taylor for extensive comments and suggestions.
\end{acknowledgements}

%
\section*{Conflict of interest}

The authors declare that they have no conflict of interest.

\bibliographystyle{spmpsci}      
\bibliography{references}   


\end{document}